\documentclass{amsart}
\usepackage{proof}
\usepackage{tikz}
\usepackage[mathscr]{eucal}
\usepackage{amssymb}

\newtheorem{theorem}{Theorem}[section]
\newtheorem{lemma}[theorem]{Lemma}
\newtheorem{cor}[theorem]{Corollary}
\newtheorem{prop}[theorem]{Proposition}
\newtheorem{defn}[theorem]{Definition}

\newtheorem{example}[theorem]{Example}

\newcommand{\deq}{\mathrel{\mathop:}=}
\newcommand{\alg}[1]{\mathbf{#1}}
\newcommand{\var}[1]{\mathsf{#1}}
\newcommand{\sipalg}[1]{\overline{\mathbf{#1}}}
\newcommand{\Bnalg}[1]{\overline{\mathbf{B}}_{#1}}

\newcommand{\Con}[1]{\mathbf{Con}\, #1}

\newcommand{\Cm}[1]{\mathbf{Cm}\, #1}
\newcommand{\cm}[1]{\mathrm{Cm}\, #1}
\newcommand{\At}[1]{\mathrm{At}(#1)}
\newcommand{\J}[1]{\mathcal{J}(#1)}

\newcommand{\um}[1]{\mathrm{I}_{#1}}
\newcommand{\dois}[1]{\mathrm{II}_{#1}}

\def\bigFOand{\mathop{\text{\rm\Large{\&}}}}

\newcommand\freefootnote[1]{%
  \let\thefootnote\relax%
  \footnotetext{#1}}%

\def\ra{\rightarrow}

\def\up#1{\mathchoice{\bigl[#1\bigr)}{[#1)}{}{}}
\def\dw#1{\mathchoice{\bigl(#1\bigr]}{(#1]}{}{}}

\def\pw{\mathscr{P}}
\def\glue#1#2#3#4{#1^#2\star{}^#3\kern-0.2ex#4}
\def\HtM{\widehat{M}}

\tikzstyle{every label}=[label distance=0pt]
\tikzstyle{bdot}[1.5]=[circle,fill,draw,thick,minimum size=#1mm,inner sep=0pt]
\tikzstyle{dot}[1.5]=[circle,draw,thick,minimum size=#1mm,inner sep=0pt]
\tikzstyle{every edge}=[draw=black,thick]

\title[Free p-algebras]{Free p-algebras revisited: an algebraic investigation of
  implication-free intuitionism} 

\author[Kowalski]{Tomasz Kowalski$^{2,3,4}$}
\author[Słomczyńska]{Katarzyna Słomczyńska$^{1}$}

\address{$^{1}$ Department of Mathematics,
University of National Education Commission, Kraków}
\email{irena.korwin-slomczynska@up.edu.pl}

\address{$^{2}$ Department of Logic, Jagiellonian University}
\email{tomasz.s.kowalski@uj.edu.pl}

\address{$^{3}$ Department of Physical and Mathematical Sciences, La Trobe
  University} 
\email{t.kowalski@latrobe.edu.au}

\address{$^{4}$ School of Historical and Philosophical Inquiry,
  The University of Queensland} 
\email{t.kowalski@uq.edu.au}

\begin{document}

\maketitle

\begin{abstract}
We give a new construction of free distributive p-algebras. Our construction
relies on a detailed description of completely meet-irreducible congruences, so
it is purely universal algebraic.
It yields a normal form theorem for p-algebra terms, simpler proofs of
several existing results, as well as a complete characterisation of structurally
complete varieties of p-algebras.
\end{abstract}

\section{Introduction}

Distributive p-algebras, which we we will call simply \emph{p-algebras},
are bounded distributive lattices endowed with a unary operation ${}^*$
satisfying the condition
$$
x\wedge y = 0 \iff x\leq y^*
$$
equivalent over bounded distributive lattices to the equations
\begin{enumerate}
\item $1^* = 0$,
\item $0^* = 1$,
\item $x\wedge(x\wedge y)^* = x\wedge y^*$.
\end{enumerate}
This basis is given, for example, in~\cite{Ber11} (Ch.~4, p.~108).   
Investigations into p-algebras peaked in 1970s (see, e.g.,
\cite{Lee70}, \cite{Lak71}, \cite{GL71}, \cite{GL72}, \cite{Lak73},
\cite{Pri75}, \cite{Ada76}, \cite{Wro76}),
they continued with less intensity into 1980s
(see, e.g., \cite{GLQ80}, \cite{Dzi85}). Then universal algebra developed
in a different direction, and in the second half of
1980s the research petered out. 

Logically, p-algebras correspond to the fragment of the intuitionistic
propositional calculus $\mathsf{INT}$, obtained by removing implication, but retaining
negation. Let us call it $\mathsf{INT}^-$. It is a natural fragment, and
we believe it deserves to be known better. One reason for this belief is somewhat
convoluted, but let us state it nevertheless. In the realm of intuitionistic
arithmetic, it is well known that adding certain arithmetical principles
unprovable over the intuitionistic base requires lifting the underlying propositional
logic somewhat: not necessarily up to the classical logic. But for principles
formulated using only conjunction, disjunction and negation, a dichotomy holds:
the underlying logic is either $\mathsf{INT}$, or $\mathsf{CPL}$. Interestingly,
both these results were shown by Rose, in a different but equivalent form,
in a single article~\cite{Ros53}.

Another reason is abstract algebraic logic. The logic 
$\mathsf{INT}^-$ can be identified with the deductive system given
by the obvious modification of the sequent calculus for $\mathsf{INT}$. 
Viewed from this perspective, p-algebras are \emph{the} algebraic semantics for 
$\mathsf{INT}^-$, in the sense of Rebagliato and Verd\'u~\cite{RV93}.
But the fragment in question fails several stronger characterisations:
in particular it is not algebraizable. We will not enter into details, but
one easy way to see the failure of algebraizability goes by showing that
the Isomorphism Theorem 3.58(ii) of Font~\cite{Fon16} fails. The
$\mathsf{INT}^-$-\emph{filters} on an arbitrary p-algebra $\alg{A}$
are the congruence classes of $1$, so algebraizability of
$\mathsf{INT}^-$ would amount to $1$-regularity of p-algebras, which fails.
P-algebras also fail to have other properties typically 
associated with algebras of logic, such as congruence
$n$-permutability.   

Finally, our original reason to look at p-algebras was 
a very nice chapter on them in the Bergman's 
textbook on universal algebra~\cite{Ber11}. It instilled in us a desire to revive the
subject. This article, as well as its sequel where we deal with quasivarieties,
grew out of that desire.

\section{Preliminaries}
We assume familiarity with the fundamentals of universal algebra. Part I
of~\cite{Ber11} will suffice. Our notation for the most part also follows 
that book, with trivial variations such as a change of font. All other notations
and notions will be introduced on the way as the need arises.

In this section we gather 
a few basic facts about the variety $\var{Pa}$ of p-algebras.
Subdirectly irreducible algebras in $\var{Pa}$ and the lattice of
subvarieties of $\var{Pa}$ were 
described in~\cite{Lee70}. 

\begin{theorem}[Lee]\label{thm:si-descr}
Let $\alg{A}\in\var{Pa}$ be subdirectly irreducible. Then $A = B\uplus\{1\}$
and $B$ is the universe of a Boolean algebra $\alg{B} = (B;\wedge,\vee,\neg,0,e)$,
where $e$ is the top element of $\alg{B}$. The order on $A$ extends
the natural order on $B$ by requiring $1>x$ for all $x\in B$. The lattice
operations are extended accordingly, and pseudo-complementation is defined by
$$
x^*\deq\begin{cases}
         \neg x & \text{ if } x\in B\setminus\{0\}\\
         1 & \text{ if } x = 0\\
         0 & \text{ if } x = 1
       \end{cases}     
$$
\end{theorem}

Following~\cite{Ber11} we write $\sipalg{B}$ for the subdirectly
irreducible p-algebra whose Boolean part is $\alg{B}$. For finite algebras we
write $\Bnalg{n}$, where $\alg{B}_n$ is the finite Boolean algebra
with $n$ atoms. Thus, $\Bnalg{0}$ is the two-element Boolean
algebra, which we will also denote by $\alg{2}$ where convenient.
We will also occasionally write $\overline{\alg{2}^s}$ instead or
$\Bnalg{s}$, to emphasise its structure. 

\begin{theorem}[Lee]\label{thm:var-descr}
The lattice of subvarieties of $\var{Pa}$ is a chain of order type $\omega+1$,
consisting of varieties
$$
\var{Pa}_{-1}\subsetneq \var{Pa}_0 \subsetneq \var{Pa}_1 \subsetneq \dots
\subsetneq \var{Pa}
$$
where $\var{Pa}_{-1}$ is the trivial variety, and
$\var{Pa}_k = V(\Bnalg{k})$ for $k\in\mathbb{N}$.
\end{theorem}
In particular $\var{Pa}_0$ is the variety $\var{BA}$ of Boolean algebras.
The variety $\var{Pa}_1$, of \emph{Stone algebras},
was studied, e.g., in~\cite{BH70}, \cite{BG71} and~\cite{Pri75}.

The following observation shows that on subdirectly irreducible algebras we have
a term operation which behaves almost like the dual of Boolean
implication.

\begin{lemma}\label{lem:si-cond}
Let $\sipalg{B}$ be a subdirectly irreducible p-algebra and let
$e$ be the unique subcover of $1$ in $\sipalg{B}$. Then
for any $a,b\in \overline{B}$ we have
$$
a\wedge b^* = 0 \iff a\leq b \text{ or } (a=1 \text{ and } b= e).
$$
\end{lemma}

\begin{proof}
In any p-algebra we have $a\wedge b^* = 0$ if and only if
$a\leq b^{**}$, so the claim holds whenever $b = b^{**}$.
Since $\alg{B}$ is subdirectly irreducible, the only element
$b$ for which $b < b^{**}$ is $b = e$. But then $a\wedge e^* = 0$ for any $a$,
and so $a\not\leq e$ only if $a = 1$.  
\end{proof}  

Identities defining $\var{Pa}_k$ were given in~\cite{Lee70}. We give different
ones below. 

\begin{lemma}\label{lem:width}
For any $m>0$, the variety $\var{Pa}_m$ is axiomatised by the single
identity
\begin{equation*}
\tag{$\mathbf{ib}_{m}$}\bigvee_{i=1}^{m+1} (x_i \wedge \bigwedge_{j\neq i} x_j^*)^* = 1. 
\end{equation*}
\end{lemma}  

\begin{proof}
It is easy to see that $\Bnalg{m+1}\not\models \mathbf{ib}_{m}$
on the valuation sending each $x_i$ to a different atom of $\alg{B}_{m+1}$,
so $\var{Pa}_k\not\models \mathbf{ib}_{m}$ for any $k>m$. Obviously,  
all infinite $\sipalg{B}$ falsify ($\mathbf{ib}_{k}$) for any $k$,
as $\Bnalg{n}\leq\sipalg{B}$ for any $n$.
  
Suppose ($\mathbf{ib}_{m}$) fails on a subdirectly irreducible algebra
$\Bnalg{k}$ for $k\leq m$, and consider a falsifying valuation $v$ with
$v(x_i) = a_i$. Then, for each $i \in\{1,\dots, m+1\}$,
we must have $a_i \wedge \bigwedge_{j\neq i} a_j^*> 0$ or equivalently
$a_i \wedge (\bigvee_{j\neq i}a_j)^* > 0$.
Since $\alg{B}_k$ is atomic, we can choose for each $i$ an
atom $b_i\leq a_i$ such that $b_i \wedge (\bigvee_{j\neq i}a_j)^* > 0$. Then,
since $(\bigvee_{j\neq i}b_j)^*\geq(\bigvee_{j\neq i} a_j)^*$, we get 
\begin{equation}\label{eq:one}
\tag{$\star$} b_i\wedge (\bigvee_{j\neq i}b_j)^* > 0
\end{equation}
where each $b_i$ is an atom. Moreover,
for each $i$  we must have $\bigvee_{j\neq i}b_j < e$ as otherwise
$(\bigvee_{j\neq i}b_j)^* = 0$ and hence
$b_i\wedge (\bigvee_{j\neq i}b_j)^* = 0$. Next, applying Lemma~\ref{lem:si-cond}
to~(\ref{eq:one}) we conclude that $b_i \not\leq \bigvee_{j\neq i}b_j$ for each $i$.
This implies that $b_i\neq b_j$ whenever $i\neq j$. But in $\alg{B}_k$ there
are at most $m$ atoms, so by pigeonhole principle we must have $b_\ell = b_r$
for some $\ell\neq r$, yielding a contradiction. 
\end{proof}

Obviously, the term-reduct of any Heyting algebra to the signature of p-algebras,
with $x^*\deq x\ra 0$, is a p-algebra. Throughout the article we will write
$\var{HA}$ for the variety of Heyting algebras, but to obtain
all p-algebras it suffices to consider \emph{Heyting algebras of height 3},
that is Heyting algebras satisfying the identity
$x\vee(x\ra(y\vee y^*)) = 1$ (several other defining identities are known).
We write $\var{HA}_3$ for this subvariety of $\var{HA}$.  
Let $\var{HA_3^r}$ be the class of term reducts of $\var{HA_3}$ to the signature
of p-algebras.

\begin{lemma}\label{lem:H3-subred}
$\var{Pa} = S(\var{HA_3^r})$.
\end{lemma}  

\begin{proof}
It is clear that $\var{Pa} \supseteq S(\var{HA_3^r})$. For the converse, take an
arbitrary $\alg{A}\in\var{Pa}$ and consider any subdirect representation
of $\alg{A}$ with subdirectly irreducible factors. By
Theorem~\ref{thm:si-descr},
each factor is of the form $\sipalg{B}$, and therefore
belongs to $\var{HA_3}$. Hence, so does the product of the factors. Since
$\alg{A}$ is subdirectly embedded (as a p-algebra) in that product, we get 
$\alg{A}\in S(\var{HA_3^r})$. 
\end{proof}  

\subsection{Universal algebraic properties}
A variety $\mathcal{V}$ with a term-definable constant $t$ is called
\emph{pointed} or $t$-\emph{pointed}. A $t$-pointed variety
is called $t$-\emph{congruence orderable}
(or simply $t$-\emph{orderable}) if for every algebra $\alg{A}\in\mathcal{V}$
and every $a,b\in A$ we have that $\Theta(a,t) = \Theta(b,t)$ implies
$a=b$, where as usual we write $\Theta(x,y)$ for the congruence generated by the
pair $(x,y)$. 

\begin{lemma}\label{lem:orderability}
Let $\alg{A}\in\var{Pa}$ and $a,b\in A$.
If $(b,1)\in\Theta(a,1)$, then $a\leq b$ in the lattice
ordering of $\alg{A}$. Therefore $\var{Pa}$ is 1-orderable. 
\end{lemma}  

\begin{proof}
Take $\alg{A}\in\var{Pa}$ and elements $a,b\in A$ such that
$(b,1)\in\Theta(a,1)$. By Lemma~\ref{lem:H3-subred} we know that
$\alg{A}\leq \alg{B}^r$ where $\alg{B}^r$ is a p-algebra reduct
of $\alg{B}\in\var{HA_3}$. Hence,
$\Theta^{\alg{A}}(a,1)\subseteq\Theta^{\alg{B}^r}(a,1)
\subseteq\Theta^{\alg{B}}(a,1)$. Therefore
$(b,1)\in\Theta^{\alg{A}}(a,1)$ implies $b\in \up{a}$,
and so by well-known properties of Heyting algebras $a\leq^{\alg{B}} b$. 
Since the lattice ordering is preserved by subreducts, we get
$a\leq^{\alg{A}} b$ proving the first statement. 

Next, suppose $\Theta(a,1) = \Theta(b,1)$. By the first statement we get
$a\leq b$ and $b\leq a$, so $a=b$ proving orderability.
\end{proof}

The next result, due to Idziak, Słomczyńska and Wroński~\cite{ISW09},
characterises subdirectly
irreducible algebras in 1-orderable varieties, generalising several such
characterisations for particular cases, most importantly, for Heyting algebras.
It is also a good place to introduce some notation we will very frequently use
henceforth. 
For an algebra $\alg{A}$ we let $\cm{\alg{A}}$ stand for the
set of completely meet-irreducible congruences, and
$\Cm{\alg{A}}$ for the poset $(\cm{\alg{A}};\subseteq)$ of completely
meet-irreducible congruences ordered by inclusion. The distinction
between the two may appear too pedantic at first, but the reader will soon see
that our construction relies crucially on a \emph{different} ordering of
$\cm{\alg{A}}$, so the distinction is needed. To avoid unnecessary pedantry,
we will not differentiate between the set of congruences on $\alg{A}$ and
the lattice of congruences on $\alg{A}$, writing $\Con{\alg{A}}$ for both.
For any $\mu\in\cm{\alg{A}}$ we write $\mu^+$ for the unique cover of
$\mu$ in $\Con{\alg{A}}$.

\begin{theorem}\label{thm:orderable-si}
Let $\alg{A}$ be an algebra from a 1-orderable variety, and let
$\mu\in\cm{\alg{A}}$. Then, the following hold:
\begin{enumerate}
\item $\alg{A}/\mu\setminus\{1/\mu\}$ has the largest element $e/\mu$,
\item $1/\mu^+ = 1/\mu\cup e/\mu$,
\item $a/\mu^+ = a/\mu$ for any $a\in A$ such that $a/\mu\notin\{1/\mu,e/\mu\}$.  
\end{enumerate}
\end{theorem}  

In other words, the monolith of any subdirectly irreducible algebra in an
orderable variety has precisely one class which is a doubleton; all other
classes are singletons. See Example~\ref{ex:si-dual}
for a picture of a subdirectly irreducible p-algebra.

A property commonly considered with respect to pointed varieties is that of
\emph{subtractivity}, introduced in Gumm, Ursini~\cite{GU84} and thoroughly
investigated in a series of articles~\cite{AU96}, \cite{AU97}, \cite{AU97a}, 
by Aglian\`o and Ursini. Recall that a $t$-pointed variety $\mathcal{V}$ is
subtractive if there is a binary term $s(x,y)$ such that $\mathcal{V}$
satisfies the equations $s(x,x) = t$ and $s(x,t) = x$. Every subtractive variety
is \emph{congruence permutable at $t$}, that is, $t/\alpha\circ\beta =
t/\beta\circ\alpha$ holds for any algebra $\alg{A}\in\mathcal{V}$ and
any $\alpha,\beta\in\Con{\alg{A}}$.

\begin{prop}\label{prop:subtractive}
The variety $\var{Pa}$ is $0$-subtractive, as witnessed by the term 
$x-y \deq (x\vee y)\wedge(x\wedge y)^*$. 
\end{prop}

\begin{proof}
We have $x-x = (x\vee x)\wedge(x\wedge x)^* = x\wedge x^* = 0$, and
$x-0 = (x\vee 0)\wedge(x\wedge 0)^* = x\wedge 0^* = x\wedge 1 = x$.
\end{proof}

\begin{prop}\label{prop:not-0-orderable}
A variety $\mathcal{V}$ of  p-algebras is $0$-orderable if and only
if $\mathcal{V}\in\{\var{Pa}_{-1}, \var{Pa}_0\}$. 
\end{prop}

\begin{proof}
Assume $\Bnalg{n}\in\mathcal{V}$ for $n>0$. Then, $\Theta(0,e) = \Theta(0,1)$,
but $e\neq 1$ showing failure of $0$-orderability.   
\end{proof}

\begin{prop}\label{prop:non-n-permut}
A variety $\mathcal{V}$ of  p-algebras is congruence permutable if and only
if $\mathcal{V}\in\{\var{Pa}_{-1}, \var{Pa}_0\}$.
Otherwise, $\mathcal{V}$ is not congruence $n$-permutable for any $n>0$. Indeed,
$\var{Pa}_k$ with $k>1$ is not $n$-permutable at $1$ for any $n>0$.
Every variety of p-algebras is permutable at $0$.
\end{prop}  

\begin{proof}
The varieties $\var{Pa_{-1}}$ and $\var{Pa_0}$ are, respectively, the trivial
variety and the variety of Boolean algebras, so they are CP. In
$\var{Pa_1}$ there are algebras $\mathbf{C}_{n+2}$ with the universe
a finite chain $0 < c_{n} < c_{n-1} < \dots <  c_1  < 1$, and
with $c_i^* = 0$ for $k\in\{1,\dots,n-1\}$. Any lattice congruence
$\Theta$ which separates $c_n$ and $0$ is a p-algebra congruence, so
the standard argument showing failure of $n$-permutability applies; in
particular, it applies to congruence classes of $1$. The last statement follows
from $0$-subtractivity.
\end{proof}

A weakening of subtractivity, called \emph{quasi-subtractivity}, was
investigated by Kowalski, Paoli and Spinks in~\cite{KPS11}. 
We state the following result without proof and refer the
reader to~\cite{KPS11} for details.

\begin{prop}\label{prop:1-quasi-subrtractive}
The variety $\var{Pa}$ is quasi-subtractive with respect to $1$.
Terms witnessing it are $x\ra y \deq (x\wedge y^*)^*$ and $\Box x \deq x^{**}$.  
\end{prop}

Summing up, p-algebras are $1$-orderable and $0$-sub\-trac\-tive, but
neither $1$-sub\-trac\-tive nor $0$-orderable. They are quasi-subtractive with
respect to both $0$ and $1$ (with different witnessing terms, of course). 
If a variety is both $c$-subtractive and $c$-orderable with respect to the same
constant $c$, then (by the proof of
Proposition~3.6 in Słomczyńska~\cite{Slo12}) it is further
$c$-\emph{regular}, that is, has the property that 
$c/\alpha = c/\beta$ implies $\alpha=\beta$ for any congruences $\alpha$ and $\beta$. 
A $c$-orderable and $c$-regular variety is \emph{Fregean}. Heyting algebras, and
some of their subreducts, notably, \emph{Brouwerian semilattices},
\emph{Hilbert algebras}, and
\emph{equivalential algebras}, are Fregean. The theory of Fregean varieties was
initiated by Pigozzi in~\cite{Pig91} and developed to its mature stage
by Idziak, Słomczyńska, Wroński in~\cite{ISW09}.

Since p-algebras are just short of being a Fregean variety, they
provide a good testing ground for studying slight
generalisations of Fregeanity. We plan to do it later: here
we focus on p-algebras, and even more specifically on free p-algebras.

We record the status of two more universal algebraic properties,
due to Gr\"atzer and Lakser~\cite{GL71},
although they will not be directly relevant for the present article.

\begin{theorem}
All varieties of p-algebras have congruence extension property (CEP). There are
precisely four varieties of p-algebras with amalgamation property (AP), namely,
$\var{Pa}_{-1}$, $\var{Pa}_{0}$, $\var{Pa}_1$ and $\var{Pa}$.   
\end{theorem}

\subsection{Topological duality}
Priestley~\cite{Pri75} gave a topological duality for p-algebras.
The objects of the dual category are Priestley spaces
$\mathbb{X} = (X,\leq,\mathcal{T})$ satisfying
the condition that for every clopen upset $U$, the set
$\dw{U}^\complement$, where ${}^\complement$ stands for the set complement,
is clopen. We will call these \emph{pp-spaces}.
The morphisms from $\mathbb{X}$ to $\mathbb{Y}$
are continuous order-preserving
maps $f\colon X\to Y$ satisfying the condition $f(\max\up{x}) = \max\up{f(x)}$
for every $x\in X$. We will call them \emph{pp-morphisms}, by analogy
to, and to avoid confusion with, p-morphisms familiar from intuitionistic
and modal logics. We use $\var{PP}$ for the category of pp-spaces with
pp-morphisms.

For a p-algebra $\alg{A}$, we let $\mathcal{F}_p(\alg{A})$ stand for the
poset of prime filters of $\alg{A}$ ordered by inclusion. For $a\in A$, let
$\widehat{a} \deq \{F\in\mathcal{F}_p(\alg{A}): a\in F\}$. Then,
the topology generated by the subbasis
$\{\widehat{a}: a\in A\}\cup\{(\widehat{b})^\complement: b\in A\}$ (with complementation
with respect to $\mathcal{F}_p(\alg{A})$, of course)  
makes $\mathcal{F}_p(\alg{A})$ into a pp-space. 
We write $\delta(\alg{A})$ for that
space, and for any homomorphism $h\colon\alg{A}\to\alg{B}$ we define
$\delta(h)\colon \delta(\alg{B})\to \delta(\alg{A})$ by $\delta(h)(F) = h^{-1}(F)$ for any
$F\in\mathcal{F}_p(\alg{B})$. So defined $\delta(h)$ is a pp-morphism.

Conversely, for a pp-space
$\mathbb{X}$, we let $\mathrm{ClopUp}(\mathbb{X})$ stand for the
set of clopen upsets of $\mathbb{X}$. The algebra
$\varepsilon(\mathbb{X}) = (\mathrm{ClopUp}(\mathbb{X}); \cup,\cap,{}^*, \emptyset,
\mathbb{X})$, with $U^* \deq \dw{U}^\complement$ is a p-algebra.
For a pp-morphism $f\colon\mathbb{X}\to\mathbb{Y}$, we define
$\varepsilon(f)\colon \varepsilon(\mathbb{Y})\to \varepsilon(\mathbb{X})$ by
$\varepsilon(f)(U) \deq \up{f^{-1}(U)}$ for any $U\in\mathrm{ClopUp}(\mathbb{Y})$.
So defined $\varepsilon(f)$ is a homomorphism of p-algebras.

Any p-algebra $\alg{A}$ is isomorphic to $\varepsilon\delta(\alg{A})$ via the map
$a\mapsto \widehat{a}$. Any pp-space $\mathbb{X}$ is isomorphic to
$\delta\varepsilon(\mathbb{X})$ via the map
$x\mapsto \{U\in\mathrm{ClopUp}(\mathbb{X}): x\in U\}$.

\begin{prop}
The categories $\var{Pa}$ (with homomorphisms) and pp-spaces (with pp-morphisms)
are dually equivalent via the contravariant functors
$\delta\colon\var{Pa} \to \var{PP}$ and $\varepsilon\colon\var{PP} \to \var{Pa}$.
\end{prop}

In a distributive lattice $\alg{L}$, every prime filter determines a unique
completely meet-irreducible congruence, so the universe
of the dual space can be identified with $\cm{\alg{L}}$ with the order given
by inclusion of 1-classes. In Heyting algebras,
this correspondence breaks down: every prime filter determines a unique
meet-irreducible congruence, but that congruence is not necessarily
\emph{completely} meet-irreducible. So, in contrast to distributive lattices,
we cannot in general work with completely meet-irreducible congruences to get a
duality for Heyting algebras. For finite Heyting algebras the correspondence
between prime filters and completely meet-irreducible congruences is restored
simply because complete meet-irreducibility coincides with meet-irreducibility.
We will see that for p-algebras the situation is just like for distributive lattices:
every prime filter determines a unique completely meet-irreducible
congruence. We will use this correspondence to build a representation of
p-algebras, and then obtaining the promised construction of free p-algebras and
normal forms for p-algebra terms.

\begin{example}\label{ex:si-dual} 
The dual $\delta(\Bnalg{n})$ of a subdirectly irreducible algebra $\Bnalg{n}$ is the
poset of height $2$ with one minimal element and $n$ maximal elements.
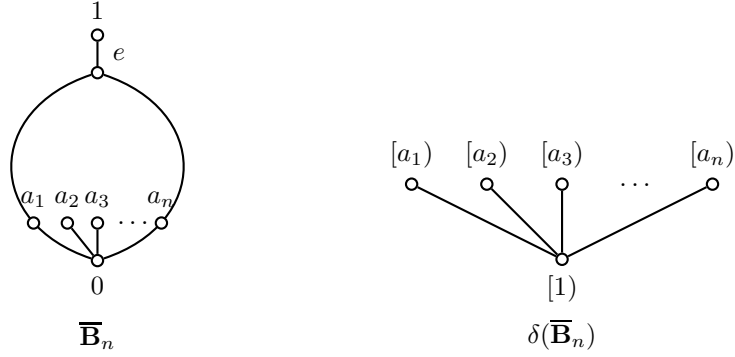
\begin{center}
\begin{tikzpicture}
\node[dot, label=below:$0$] (v0) at (0,0) {};
\node[dot, label=above:$1$] (v1) at (0,3) {};
\node[dot, label=30:$e$] (v2) at (0,2.5) {};
\draw[thick] (v0).. controls (-1.5,0.5) and (-1.5,2) .. (v2)
(v0) .. controls (1.5,0.5) and (1.5,2) .. (v2) (v1)--(v2);
\node[dot, fill=white, label=above:$a_1$] (v3) at (-0.85,0.5) {};
\node[dot, label=above:$a_2$] (v4) at (-0.4,0.5) {};
\node[dot, label=above:$a_3$] (v5) at (0,0.5) {};
\node[] (v6) at (0.5,0.5) {$\dots$};
\node[dot, fill=white, label=above:$a_n$] (v7) at (0.85,0.5) {};
\draw[thick] (v0)--(v4) (v0)--(v5);
\node[] (v7) at (0,-1) {$\Bnalg{n}$};
\end{tikzpicture}  
\hspace{2cm}
\begin{tikzpicture}
\node[dot, label=below:$\up{1}$] (v0) at (0,0) {};
\node[dot, label=above:$\up{a_1}$] (v1) at (-2,1) {};
\node[dot, label=above:$\up{a_2}$] (v2) at (-1,1) {};
\node[dot, label=above:$\up{a_3}$] (v3) at (0,1) {};
\node[dot, label=above:$\up{a_n}$] (v4) at (2,1) {};
\node[] (v5) at (1,1) {$\dots$};
\path[draw,thick] (v0)--(v1) (v0)--(v2) (v0)--(v3) (v0)--(v4);
\node[] (v6) at (0,-1) {$\delta(\Bnalg{n})$};
\end{tikzpicture}
\end{center}
\end{example}

It is well known that in general there exist 
p-algebras which are not reducts of Heyting algebras, and bounded distributive
lattices which are not reducts of p-algebras. However, in the finite case
the three classes coincide, that is, every finite bounded distributive lattice
is a reduct of a p-algebra which in turn a reduct of a Heyting algebra.

For our purposes, the duality described above is less useful than another form of
representation, whose building blocks are completely meet-irreducible
congruences. Although not a duality, this representation will allow us to
construct free p-algebras in a neater and more systematic way than
in existing constructions by Urquhart in~\cite{Urq73}, by Berman and Dwinger
in~\cite{BD75}, and by Katri\v{n}\'{a}k in~\cite{Kat98}.

\section{Completely meet irreducible congruences}

Let $\alg{A}$ be an arbitrary algebra. The map
$M\colon \mathbf{Con}\,\alg{A} \to \mathrm{Up}(\Cm{\alg{A}})$ 
given for any $\alpha\in\mathbf{Con}\,\alg{A}$ by
$$
M(\alpha) \deq \{\mu\in\cm{\alg{A}}: \alpha\subseteq\mu\}
$$ 
is injective, by Birkhoff Theorem. If $\alg{A}$ is finite and congruence-distributive,
then $M$ is a dual lattice isomorphism. For orderable algebras (not necessarily
finite, or congruence distributive) we have another
useful map. Namely, if $\alg{A}$ is $1$-orderable,
then the map $\HtM\colon \alg{A}\to \mathrm{Up}(\Cm{\alg{A}})$,
given by
$$
\HtM(a)\deq M(\Theta(a,1))
$$
is injective. For such an $\alg{A}$ it is well known that
$\mathrm{Up}(\Cm{\alg{A}})$ carries a natural Heyting algebra structure,
namely $(\mathrm{Up}(\Cm{\alg{A}}); \cup,\cap, \ra, \emptyset,\cm{\alg{A}})$
with $U\ra W\deq \dw{U\setminus W}^\complement$ where ${}^\complement$ stands for the
set complement, is a Heyting algebra. Hence, it is also a p-algebra, with
$U^* = U\ra\emptyset$, or explicitly,
$U^* \deq \dw{U}^\complement$.

From now on, we assume $\alg{A}$ is a p-algebra, and we begin a series of observations
on the the structure of $\Cm{\alg{A}}$. Since any $\mu\in\cm{\alg{A}}$ has
a unique cover in $\Con{\alg{A}}$ we will write $\mu^+$ for that cover.

\begin{lemma}\label{lem:mu-plus}
Let $\alg{A}\in\var{Pa}$ and $\mu\in\cm{\alg{A}}$.  Then,
\begin{enumerate}
\item $1/\mu$ is a prime filter.
\item $1/\mu^+$ is a disjoint union of two classes: $1/\mu$ and $e/\mu$, where
$e$ is any element of $A$ such that $e/\mu$ is the unique subcover of
$1/\mu$ in $\alg{A}/\mu$. 
\item For any $a\notin 1/\mu^+$ we have
$a/\mu^+ = a/\mu = a^{**}/\mu = a^{**}/\mu^+$.
\item $\alg{A}/\mu^+$ is a Boolean algebra.
\end{enumerate}
\end{lemma}  

\begin{proof}
For any $\alpha\in\Con{\alg{A}}$ the class $1/\alpha$ is a filter, so
for (1) we only need to show that $1/\mu$ is prime. Take $a,b\in A$ with
$a\vee b\in 1/\mu$, that is, $(a\vee b, 1)\in\mu$. Since
$\Theta(a\vee b,1) = \Theta(a,1)\cap\Theta(b,1)$, we have
$\Theta(a,1)\cap\Theta(b,1)\leq \mu$, so by complete meet-irreducibility of
$\mu$ (and congruence distributivity) we obtain $\Theta(a,1)\leq\mu$ or 
$\Theta(b,1)\leq\mu$. Hence, $a\in 1/\mu$ or $b\in 1/\mu$ as required.

For (2), first observe that the algebra $\alg{A}/\mu$ is subdirectly
irreducible. Thus $\alg{A}/\mu$ is isomorphic to $\Bnalg{m}$ for some $m\geq
0$. Pick $e\in A$ such that 
$e/\mu$ is the unique subcover of $1/\mu$ in $\alg{A}/\mu$, so that
$e^*/\mu = 0/\mu$. 
The quotient congruence $\mu/\mu^+$ is the monolith of $\alg{A}/\mu$, and so
by Theorem~\ref{thm:orderable-si} it is the identity on $\alg{A}/\mu$ except 
$(1/\mu)/\mu^+ = \{1/\mu, e/\mu\}$.

For (3) and (4), again using the fact that $\alg{A}/\mu\cong \Bnalg{m}$,
we get the only $\mu$-class $a/\mu$
for which $a/\mu \neq a^{**}/\mu$ is $e/\mu$, as $e^{**}/\mu = 1/\mu$.
Since $\mu^+\succ \mu$, we get $\alg{A}/\mu^+\cong\alg{B}_m$.
\end{proof}

The element $e\in A$ in the proof above is in general not unique, but
a suitable $e$ exists for any particular $\mu$, so from now on
we will use $e_\mu$ to name any such element. Using this convention
we will write $1/\mu^+ = 1/\mu\cup e_\mu/\mu$ with no danger of confusion.

\begin{lemma}\label{lem:homomorphism}
Let $\alg{A}\in\var{Pa}$. Then, $\HtM\colon\alg{A}\to
\mathrm{Up}(\Cm{\alg{A}})$ is a homomorphism.
Therefore, $\alg{A}\cong \HtM(\alg{A})\leq\mathrm{Up}(\Cm{\alg{A}})$. 
\end{lemma}  

\begin{proof}
Take any $a,b\in A$, and any $\mu\in\cm{\alg{A}}$.
It is immediate that $(a\wedge b)/\mu = 1/\mu$ if and only if
$a/\mu = 1/\mu$ and $b/\mu = 1/\mu$, and by Lemma~\ref{lem:mu-plus}
we also have that $(a\vee b)/\mu = 1/\mu$ if and only if
$a/\mu = 1/\mu$ or $b/\mu = 1/\mu$.

We will now show that $\HtM(a^*) = \dw{\HtM(a)}^\complement$.
For the left-to-right inclusion, take any $\mu\in\cm{\alg{A}}$ such that
$(a^*,1)\in \mu$. Suppose 
$\mu\in\dw{\HtM(a)}$, so that $\mu\leq\varphi$ for some
$\varphi\in\HtM(a)$. Then $(a,1)\in\varphi$ and $(a^*,1)\in \varphi$, and thus
$(0,1)\in\varphi$ so $\varphi$ is the full congruence, contradicting
$\varphi\in\cm{\alg{A}}$. 

For the converse, take $\mu\in \dw{\HtM(a)}^\complement$. First we show that
$(a^*,1)\in \mu^+$. For if $(a^*,1)\notin \mu^+$, then
for some $\varphi\in M(\mu^+)$ we have $(a^*,1)\notin\varphi$. 
But $\alg{A}/\varphi\cong\mathbf{2}$ and hence
$(a,1)\in\varphi$, so we get
$\mu<\mu^+\leq \varphi\in \HtM(a)$, contradicting  
$\mu\notin \dw{\HtM(a)}$. Therefore $(a^*,1)\in \mu^+$.
As $1/\mu^+ = 1/\mu\cup e_\mu/\mu$, we have
two cases to consider. If $a^*\in e_\mu/\mu$, then
$a^*/\mu = e_\mu^{**}/\mu = 1/\mu$ implying $(e_\mu,1)\in\mu$ which is a
contradiction. Hence, $a^*\in 1/\mu$, that is,
$\mu\in\HtM(a^*)$ as required.

Injectivity of $\HtM$, required for the second statement,
is an immediate consequence of $1$-orderability.
\end{proof}

For any finite Heyting algebra $\alg{A}$
the map $\HtM$ gives an order-isomorphism
$\HtM(\alg{A})\cong\mathrm{Up}(\Cm{\alg{A}})$.
In finite p-algebras, we have nearly the same: the order is preserved but not
reflected. Our next goal is to repair this by finding a suitable
ordering on $\cm{\alg{A}}$. 
It will be crucial for the description of free algebras to come. 

Immediately by the description of subdirectly irreducibles we have that 
$\Cm{\alg{A}}$ has two `storeys': one comprising all congruences whose quotients
are isomorphic to $\Bnalg{0}$ (that is, $\alg{2}$), and the other comprising the rest.

\begin{defn}\label{def:storeys}
Let $\alg{A}\in\var{Pa}$. Put
\begin{align*}
\mathrm{I}_\alg{A}
&\deq
\{\mu\in \cm{\alg{A}}: \alg{A}/\mu\cong \Bnalg{0}\}
= \{\mu\in \cm{\alg{A}}: \mu^+ = \mathbf{1}^\alg{A}\}, \\
\mathrm{II}_\alg{A}
&\deq
\{\mu\in \cm{\alg{A}}: \alg{A}/\mu\cong \Bnalg{n} 
\text{ for } n>0\}
= \{\mu\in \cm{\alg{A}}: M(\mu^+)\subseteq \mathrm{I}_\alg{A}\}.
\end{align*}
\end{defn}
Immediately, we conclude that
$\cm{\alg{A}} = \mathrm{I}_\alg{A} \cup \mathrm{II}_\alg{A}$.
Moreover, $\mathrm{I}_\alg{A}$ is precisely the set of all maximal members of
$\Cm{\alg{A}}$ and $\mathrm{II}_\alg{A}$ is the set of
all non-maximal members of $\Cm{\alg{A}}$, which 
coincides with the set of all minimal but not maximal members of
$\Cm{\alg{A}}$.

\begin{lemma}\label{lem:I-is-regular}
Let $\alg{A}\in\var{Pa}$, let $\varphi\in\mathbf{Con}\,\alg{A}$.
If $M(\varphi)\subseteq \mathrm{I}_\alg{A}$, then
for any $a,b\in A$ we have
\begin{align*}
  (a,b) \in \varphi &\iff \bigl((a^*\wedge b^*)\vee(a\wedge b), 1\bigr)\in\varphi\\
                    &\iff \bigl((a\vee b)\wedge (a\wedge b)^*, 0\bigr)\in\varphi.
\end{align*}                      
\end{lemma}  

\begin{proof}
Let $M(\varphi) = \{\mu_i: i\in I\}$. Then $\alg{A}/{\mu_i}\cong\alg{2}$ for
every $i\in I$. Since $\varphi = \bigcap_{i\in I}\mu_i$, we have
\begin{align*}
(a,b)\in \varphi
&\iff \forall i\in I: (a,b)\in \mu_i\\ 
&\iff \forall i\in I: \bigl((a^*\wedge b^*)\vee(a\wedge b), 1\bigr)\in \mu_i
\iff \bigl((a^*\wedge b^*)\vee(a\wedge b), 1\bigr)\in\varphi\\
&\iff \forall i\in I: \bigl((a\vee b)\wedge(a\wedge b)^*, 0\bigr)\in \mu_i
\iff \bigl((a\vee b)\wedge(a\wedge b)^*, 0\bigr)\in\varphi 
\end{align*}
where the second and third line are equivalent to one another
by the usual De~Morgan laws for $\alg{2}$. 
\end{proof}
On the few occasions we need it, we will write $a\cdot b$ as a shorthand for
$(a^*\wedge b^*)\vee(a\wedge b)$, and
$a-b$ as a shorthand for $(a\vee b)\wedge(a\wedge b)^*$. 
Note that $a\cdot b$ interpreted in a Boolean
algebra is the Boolean equivalence, and $a-b$ is the symmetric difference
(or addition mod 2). Neither the Heyting equivalence, nor its dual are
term-definable in p-algebras. 

As we saw, $\var{Pa}$ is not 1-regular, but 
Theorem~\ref{thm:cm-regular} will show that $1$-regularity holds for completely
meet-irreducible congruences, as in the case of distributive lattices.
Vestiges of $1$-regularity for all congruences are
summarised below, as a corollary of Lemma~\ref{lem:I-is-regular}.

\begin{cor}\label{cor:weak-1-reg}
Let $\alg{A}\in\var{Pa}$, let $\varphi,\psi\in\mathbf{Con}\,\alg{A}$.
If $M(\varphi)\subseteq \mathrm{I}_\alg{A}$ and $M(\psi)\subseteq
\mathrm{I}_\alg{A}$, then $1/\varphi = 1/\psi$ implies $\varphi = \psi$.
\end{cor}

\begin{theorem}\label{thm:cm-regular}
Let $\alg{A}\in\var{Pa}$ and let $\mu,\nu\in\cm{\alg{A}}$. 
Then, $1/\mu = 1/\nu$ implies $\mu = \nu$.
\end{theorem}

\begin{proof}
Assume $1/\mu = 1/\nu$. We have three cases to consider.

(1) $\mu,\nu\in \mathrm{I}_\alg{A}$. Then, $\alg{A}/\mu\cong
\Bnalg{0}\cong\mathbf{2}\cong\alg{A}/\nu$ 
and so $A = 1/\mu\cup 0/\mu = 1/\nu\cup 0/\nu$.
By assumption, $1/\mu = 1/\nu$, so
$0/\mu = 0/\nu$ as well; hence $\mu = \nu$.

(2) $\mu\in \mathrm{I}_\alg{A}$, $\nu\in \mathrm{II}_\alg{A}$.
Then $\alg{A}/\mu\cong \Bnalg{0}\cong\mathbf{2}$ and
$\alg{A}/\nu\cong \Bnalg{s}$ for some $s > 1$. Consider
$e_\nu\in A$; recall that by our convention $e_\nu/\nu$ is the unique subcover
of $1/\nu$. Hence, $1/\nu = (e_\nu/\nu)^{**} = e_\nu^{**}/\nu$,
that is, $(e_\nu^{**},1)\in \nu$, so by
assumption $(e_\nu^{**},1)\in\mu$. Since $A = 1/\mu\cup 0/\mu$, we get that if
$e_\nu\in 0/\mu$ then $e_\nu^{**} \in 0/\mu$ contradicting nontriviality of
$\alg{A}/\mu$. So, $e_\nu\in 1/\mu$, but then $e_\nu\in 1/\nu$, contradicting
$e_\nu/\nu\prec 1/\nu$. It follows that this case cannot occur.

(3) $\mu,\nu\in \mathrm{II}_\alg{A}$. Thus,
$\alg{A}/\nu\cong \Bnalg{s}$ and $\alg{A}/\nu\cong \Bnalg{r}$
for some $s,r > 1$. We claim that $1/\nu^+ = 1/\mu^+$.
Pick any $a\in 1/\nu^+$ and recall that $1/\nu^+ = 1/\nu\cup e_\nu/\nu$.
If $a\in 1/\nu$, then by assumption $a\in 1/\mu$ and so
$a\in 1/\mu^+$. If $a\in e_\nu/\nu$, then $a^{**}/\nu = 1/\nu$ and so by
assumption $a^{**}\in 1/\mu$. By Lemma~\ref{lem:I-is-regular} we get
$a\in 1/\mu^+$ showing $1/\nu^+\subseteq 1/\mu^+$.
The converse inclusion follows by symmetry,
proving the claim.

By Corollary~\ref{cor:weak-1-reg} we conclude $\nu^+ = \mu^+$.
Now let $(a,b)\in\mu$ and to get a contradiction suppose $(a,b)\notin\nu$.
So we have $(a,b)\in\nu^+\setminus\nu$, and this implies
$\{a/\nu, b/\nu\} = \{1/\nu, e_\nu/\nu\}$. Without loss of generality,
suppose $a/\nu = 1/\nu$, $b/\nu = e_\nu/\nu$. As $1/\nu = 1/\mu$ by assumption,
we have $(a,1)\in\mu$ and since $(a,b)\in\mu$ we get $(b,1)\in\mu$.
This in turn implies $(b,1)\in\nu$ and 
we obtain $(1,e_\nu)\in\nu$, that is, $e_\nu\in 1/\nu$,
yielding the desired contradiction. 
\end{proof}

Let $\Phi\colon \Cm{\alg{A}}\to \mathcal{F}_p(\alg{A})$ be the map given by
$\Phi(\mu)\deq 1/\mu$. By Lemma~\ref{lem:mu-plus}, the class
$1/\mu$ is a prime filter, so $\Phi$ is well defined.

\begin{lemma}\label{lem:I-image}
Consider the image $\Phi(\mathrm{I}_\alg{A})$. We have
$$
\Phi(\mathrm{I}_\alg{A}) = \{F\in \mathcal{F}_p(\alg{A}): 
\forall a\in A: a^{**}\in F\implies a\in F\}.
$$
\end{lemma}

\begin{proof}
For the left-to-right inclusion, let $\mu\in\mathrm{I}_\alg{A}$,
$a\in A$ and $a^{**}\in 1/\mu$. Then $\alg{A}/\mu\cong\alg{2}$ and so
$1/\mu = a^{**}/\mu = (a/\mu)^{**} = a/\mu$; hence $a\in 1/\mu$.

For the
converse, take $F\in\mathcal{F}_p(\alg{A})$ such that
$a^{**}\in F$ implies $a\in F$ for any $a\in A$. The map
$h_F\colon \alg{A}\to\alg{2}$ given by $h_F(a) = 1$ if $a\in F$ and
$h_F(a) = 0$ otherwise, is a homomorphism of bounded distributive lattices. 
Pick any $a\in A$; since $(a\vee a^*)^{**} = 1$ for any $a\in A$, we have
$a\vee a^*\in F$. Since $F$ is prime, $a\in F$ or $a^*\in F$. In either case,
$h_F(a^*) = h_F(a)^*$ and so $h_F$ is a homomorphism of p-algebras.
Thus, $\alg{A}/\ker(h_F)\cong\alg{2}$ and so $\ker(h_F)\in\mathrm{I}_\alg{A}$
and $\Phi(\ker(h_F)) = F$.
\end{proof}

\begin{theorem}\label{thm:cm-into-fp}
Let $\alg{A}\in\var{Pa}$, and let
$\Phi\colon \Cm{\alg{A}}\to \mathcal{F}_p(\alg{A})$ be given by
$\Phi(\mu) = 1/\mu$. Then $\Phi$ is an order-preserving bijection, but not an
order-isomorphism. 
\end{theorem}

\begin{proof}
For any $\mu\in \Cm{\alg{A}}$ the class $1/\mu$ is a prime filter by 
Lemma~\ref{lem:mu-plus}, so $\Phi$ is well defined, and trivially
order-preserving. By Theorem~\ref{thm:cm-regular} we get that
$\Phi$ is injective. It remains to show surjectivity. 

Take $F\in \mathcal{F}_p(\alg{A})$. By Lemma~\ref{lem:I-image} we can assume
there is some $a\in A$ with $a^{**}\in F$ but $a\notin F$.
Define $\overline{F}\deq\{a\in A: a^{**}\in F\}$. It is easily shown that
$\overline{F}$ is a proper filter, strictly containing $F$. Note that
$b\vee b^*\in \overline{F}$ for any $b\in A$, since
$(b\vee b^*)^{**} = (b^*\wedge b^{**})^* = 1$. For any $a\notin\overline{F}$  
let $G_a$ be a prime filter extending $\overline{F}$ with $a\notin G_a$.
Then $G_a$ satisfies the condition from Lemma~\ref{lem:I-image}, for
if $b^{**}\in G_a$, then as $b\vee b^*\in \overline{F}$, we get
$b = b^{**}\wedge(b\vee b^*) \in G_a$.  Hence $G_a = \Phi(\mu_a)$ for some
$\mu_a\in\mathrm{I}_\alg{A}$, and therefore
$\overline{F} = \bigcap\{G_a: a\notin \overline{F}\}$. 
Let $\alpha\deq\bigcap\{\mu_a: a\notin\overline{F}\}$. Then, obviously,
$\alg{A}/\alpha\leq \prod_{a\notin\overline{F}} \alg{A}/\mu_a$, and since
$\alg{A}/\mu_a\cong\alg{2}$ we obtain that $\alg{A}/\alpha$ is a Boolean
algebra. Moreover, $1/\alpha = \overline{F}$, for we have
$$
  b\in 1/\alpha \iff (b,1)\in\alpha
                  \iff (b,1)\in\bigcap_{a\notin\overline{F}}\mu_a
\iff b\in \bigcap_{a\notin\overline{F}} 1/\mu_a\iff b\in\overline{F}.
$$

Define an equivalence relation $\sim$ on $A$, putting
$$
a\sim b \iff \bigl(\{a,b\}\cap \overline{F}=\emptyset \text{ and } (a,b)\in \alpha\bigr)
\text{ or } \{a,b\}\subseteq\overline{F}\setminus F  \text{ or } \{a,b\}\subseteq F.
$$
Equivalence classes of $\sim$ are precisely the congruence classes of $\alpha$,
except $1/\alpha$ which is partitioned into $F$ and $\overline{F}\setminus F$.
We show that $\sim$ is a congruence.

(1) Suppose $a\sim b$ and consider $a^*$ and $b^*$. By definition
$(a^*,b^*)\in\alpha$, and if $a^*\not\sim b^*$ then
without loss of generality $a^*\in\overline{F}\setminus F$, which immediately 
contradicts the definition of $\overline{F}$, as $a^* = a^{***}$. Hence
$\{a^*,b^*\}\subseteq F$, that is, $a^*\sim b^*$.

(2) Suppose $a\sim b$, $c\sim d$ and consider $a\wedge c$ and $b\wedge d$.
By definition of $\sim$ and the symmetries of the situation, the only non-obvious
case could arise with $a\wedge c\in \overline{F}\setminus F$ and
$b\wedge d\in F$. But this would imply $\{a,c\}\subseteq F$ and so
$a\wedge c\in F$, contradicting $a\wedge c\in \overline{F}\setminus F$
so this case is impossible. Hence, $(a\wedge c) \sim (b\wedge d)$.

(3) Suppose $a\sim b$, $c\sim d$ and consider $a\vee c$ and $b\vee d$.
Again, the only non-obvious cases could happen with
$a\vee c\in \overline{F}\setminus F$ and $b\vee d\in F$.
By primeness of $F$ then $b\in F$ or $d\in F$, so without loss of generality
suppose $b\in F$. Then $a\in F$ and so $\{a\vee c,b\vee d\}\subseteq F$. Hence
$(a\vee c) \sim (b\vee d)$.

To finish the proof of surjectivity of $\Phi$ it suffices to show
that $\alg{A}/_\sim$ is subdirectly irreducible.
This however is immediate: by definition of $\sim$ we see that
$\overline{F}\setminus F$ is the largest element of $\alg{A}/_\sim$ different
from $1/_\sim$; hence $\alg{A}/_\sim$ is isomorphic to $\sipalg{B}$ where
$\alg{B} = \alg{A}/\alpha$. 

Finally, Example~\ref{ex:not-iso} below shows that $\Phi$ in not an order-isomorphism:
$\mathcal{F}_p(\alg{C}_4)$ is a chain, whereas $\Cm{\alg{C}_4}$ is not.
\end{proof}

The next result is a corollary of Theorem~\ref{thm:cm-into-fp} and the duality
for the finite case.

\begin{cor}\label{cor:up-cm-charact}
Let $\alg{A}\in\var{Pa}$ be finite, and let $U\subseteq\Cm{\alg{A}}$.
Then
$$
U\in\HtM(\alg{A}) \iff \Phi(U)\in \mathrm{Up}(\mathcal{F}_p(\alg{A}))
$$
where  and $\Phi(U) = \{1/\mu: \mu\in U\}$, of course.
\end{cor}

\begin{example}\label{ex:not-iso}
Let $\mathbf{C}_4$ be the p-algebra whose lattice reduct is a 4-element chain.
The poset $\Cm{\alg{C}_4}$ has 3 elements ordered into a
$\Lambda$ shape, and the algebra $\mathrm{Up}(\Cm{\alg{C}_4})$
has 5 elements, with two subcovers of $1$. The embedding $\HtM$ maps
$\mathbf{C}_4$ to one of two possible 4-element chains in
$\mathrm{Up}(\Cm{\alg{C}_4})$. 
\begin{center}
\begin{tikzpicture}
\node[dot, label=right:$0$] (v0) at (0,0) {};
\node[dot, label=right:$a$] (v1) at (0,0.5) {};
\node[dot, label=right:$b$] (v2) at (0,1) {};
\node[dot, label=right:$1$] (v3) at (0,1.5) {};
\path[draw,thick] (v0)--(v1)--(v2)--(v3);
\node (v4) at (0.1,-1) {$\alg{C}_4$};
\end{tikzpicture}
\hspace{1.5cm}
\begin{tikzpicture}
\node[dot, label=right:{$\Theta(a,1)$}] (v0) at (0,1) {};
\node[dot, label=left:{$\Theta(a,b)$}] (v1) at (-0.5,0) {};
\node[dot, label=right:{$\Theta(b,1)$}] (v2) at (0.5,0) {};
\path[draw,thick] (v0)--(v1) (v0)--(v2);
\node (v3) at (0,-1) {$\mathbf{Cm}\,\mathbf{C}_4$};
\end{tikzpicture}
\hspace{1.5cm}
\begin{tikzpicture}
\node[dot, label=right:$\emptyset$] (v0) at (0,0) {};
\node[dot, label=right:$\up{\Theta(a,1)}$] (v1) at (0,0.5) {};
\node[dot, label=left:{$\up{\Theta(a,b)}$}] (v2) at (-0.5,1) {};
\node[dot, label=right:{$\up{\Theta(b,1)}$}] (v3) at (0.5,1) {};
\node[dot] (v4) at (0,1.5) {};
\path[draw,thick] (v0)--(v1)--(v2)--(v4) (v1)--(v3)--(v4);
\node (v4) at (0,-1) {$\mathrm{Up}(\Cm{\alg{C}_4})$};
\end{tikzpicture}  
\end{center}
\end{example}

To upgrade the map $\Phi$ to order-isomorphism, we now introduce a tighter order
on completely meet-irreducible congruences of any $\alg{A}\in\var{Pa}$.

\begin{defn}\label{def:cm-order}
Let $\alg{A}\in\var{Pa}$. For any $\mu,\nu\in\cm{\alg{A}}$ we put
$$  
\mu\leq^{\Cm{}}\nu \iff 1/\mu\subseteq 1/\nu.
$$
\end{defn}
It follows from Theorem~\ref{thm:cm-regular} that 
$\leq^{\Cm{}}$ is indeed an ordering relation, so the definition above is sound.
Example~\ref{ex:not-iso} shows that the new order is strictly tighter than
inclusion.

\begin{theorem}\label{thm:cm-summary}
Let $\alg{A}\in\var{Pa}$ and $\mu,\nu\in\cm{\alg{A}}$. Then, the following hold.
\begin{enumerate}
\item If $\mu\subseteq \nu$, then $\mu\leq^{\Cm{}}\nu$.
\item If $\mu,\nu\in\um{\alg{A}}$, then $\mu\subseteq \nu \iff \mu\leq^{\Cm{}}\nu$. 
\item If $\mu\in\dois{\alg{A}}, \nu\in\um{\alg{A}}\implies (\mu\subseteq \nu \iff
  \mu\leq^{\Cm{}}\nu)$,   
\item If $\alg{A}$ is finite, then
\begin{enumerate}
\item $\alg{A} \cong \mathrm{Up}(\cm{\alg{A}};\leq^{\Cm{}})$, and
\item ${id}\colon (\cm{\alg{A}},\subseteq)\to (\cm{\alg{A}},\leq^{\Cm{}})$ is a surjective pp-morphism.
\end{enumerate}  
\end{enumerate}
\end{theorem}  

\begin{proof}
(1) is immediate from definition. For (2), note that if
$\mu,\nu\in\um{\alg{A}}$ then $\alg{A}/\mu\cong\alg{2}\cong\alg{A}/\nu$, and
then using Lemma~\ref{lem:I-is-regular} and properties of p-algebras, we easily
obtain that $1/\mu\subseteq 1/\nu$ implies $\mu=\nu$.

For (3) the forward
direction is obvious; to show the backward direction note first that
$1/\mu^+\subseteq 1/\nu$. By Theorem~\ref{thm:orderable-si}(2) we have that
$1/\mu^+ = 1/\mu\cup e_\mu/\mu$. Let $a\in e_\mu/\mu$; then
$a^{**}/\mu = 1/\mu$ and so $a^{**} \in 1/\mu$. But $(a,a^{**})\in\nu$ and thus
$a\in 1/\nu$. Hence, $1/\mu^+\subseteq 1/\nu$. Now we will show that
$\mu\subseteq\nu$. To this end, let $(a,b)\in\mu$. Then $(a,b)\in\mu^+$
and so using Lemma~\ref{lem:I-is-regular} we get that
$(1, a\cdot b)\in\mu^+$, so  $(1, a\cdot b)\in\nu$ and therefore
$(a,b)\in\nu$. 

For (4) part (a), by Theorem~\ref{thm:cm-into-fp} and
Definition~\ref{def:cm-order} we get that $(\mathcal{F}_p(\alg{A});\subseteq)$
and $(\cm{\alg{A}};\leq^{\Cm{}})$ are order-isomorphic, so the claim follows by
duality. Finally, (4) part (b) is immediate by (1), (2), (3) and the fact that the
for pp-morphisms of finite pp-spaces topological conditions are irrelevant, as
the topology is discrete.
\end{proof}  

\section{Join-irreducible elements}

Although further study of representations of arbitrary p-algebras in topologised
posets of their completely meet-irreducible congruences seems interesting, we will not
pursue it in this article. By local finiteness, the finite case is all
we need for the description of free algebras.

Just like in finite distributive lattices and finite Heyting algebras, for a finite
p-algebra $\alg{A}$ there is a dual order-isomorphism between
$\mathcal{F}_p(\alg{A})$ and the set $\mathcal{J}(\alg{A})$ of join-irreducible
elements of $A$, given by $\mathcal{J}(\alg{A})\ni a\mapsto \up{a}\in
\mathcal{F}_p(\alg{A})$, with the order $a\leq b \iff \up{b}\subseteq\up{a}$.  
As an immediate corollary of Theorem~\ref{thm:cm-into-fp} we obtain the result
below.

\begin{cor}\label{cor:cm-into-jirr}
Let $\alg{A}\in\var{Pa}$ be finite, and let
$\Psi\colon \Cm{\alg{A}}\to \mathcal{J}(\alg{A})$ be given by
$\Psi(\mu) = \bigwedge\Phi(\mu)$. Then $\Psi$ is an order-inverting bijection, but
not a dual order-isomorphism.
\end{cor}

\begin{lemma}\label{lem:star-for-join-irr}
Let $\alg{A}\in\var{Pa}$ be finite. If $a\in \J{\alg{A}}$ is not an
atom, then $(a_-)^* = a^*$, where $a_-$ is the unique subcover of $a$.
\end{lemma}

\begin{proof}
Obviously $a\wedge (a_-)^* \leq a$, and since $a\in \J{\alg{A}}$
we have either $a\wedge (a_-)^* = a$ or $a\wedge (a_-)^* \leq a_-$.
If the former, then $0 = a\wedge (a_-)^* \wedge a_- = a\wedge a_- = a_-$
contradicting the assumption that $a$ is not an atom. Hence, we have
$a\wedge (a_-)^* \leq a_-$, and so
$a\wedge (a_-)^* \leq a_-\wedge (a_-)^* = 0$. Then by properties of $*$ it
follows that $(a_-)^*\leq a^*$. On the other hand,
as $a_-\leq a$, we get $a^*\leq (a_-)^*$ completing the proof.
\end{proof}

\begin{lemma}\label{lem:atoms}
Let $\alg{A}\in\var{Pa}$ be finite, and let $\mu\in\cm{\alg{A}}$. Then
$$  
\mu\in\mathrm{I}_\alg{A}\iff 1/\mu = \up{a} \text{ for some atom } a.
$$
\end{lemma}

\begin{proof}
For the forward direction, it is clear that $1/\mu = \up{a}$ for
some $a\in\J{\alg{A}}$. Then $a_-\notin 1/\mu$, and since
$\mu\in\mathrm{I}_\alg{A}$, we have $\alg{A}/\mu \cong\mathbf{2}$;
therefore $a_-/\mu = 0/\mu = a^*/\mu$, and $(a_-)^*/\mu = 1/\mu$.
This means $a^*\neq (a_-)^*$, so by Lemma~\ref{lem:star-for-join-irr}
$a$ is an atom. 

For the backward direction, suppose $\mu\in\mathrm{II}_\alg{A}$.
Then, there is some $\nu > \mu$ with $\nu\in\mathrm{I}_\alg{A}$
and $1/\nu = \up{b}$ for some $b\in\J{\alg{A}}\setminus\{a\}$.
Then $\up{a} = 1/\mu \subsetneq 1/\nu = \up{b}$ and so $0 < b < a$,
contradicting the fact that $a$ is an atom. 
\end{proof}  

Atoms will be important in the considerations to come, so
we introduce a common shorthand $\At{\alg{A}}$ for the set of atoms of $\alg{A}$.

\begin{lemma}\label{lem:atoms-and-cms}
Let $\alg{A}\in\var{Pa}$ be finite, let $a\in\At{\alg{A}}$. Then
\begin{enumerate}
\item $\HtM(a) = \{\mu\}$  and $\mu = \Theta(a,1)$,
\item $\HtM(a^*) = (\um{\alg{A}}\setminus \{\mu\})
\cup\{\nu\in\dois{\alg{A}}: \nu\not\subseteq \mu\}$,
\item $\HtM(a^{**}) = \{\mu\} \cup\{\nu\in\dois{\alg{A}}: \nu^+ = \mu\}$.
\end{enumerate}  
\end{lemma}  

\begin{proof}
Since $\At{\alg{A}}\subseteq\J{\alg{A}}$, by Lemma~\ref{lem:atoms} we have
$\up{a} = 1/\mu$  for some $\mu\in\um{\alg{A}}$.  Clearly,
$\mu\in\HtM(a)$. Suppose $\nu\in\HtM(a)$; then $1/\nu = \up{b}$ for some
$b\in\J{\alg{A}}$ such that $a\in\up{b}$, so $b\leq a$. Since
$a\in\At{\alg{A}}$ and $0\notin\J{\alg{A}}$, we get $a = b$. Hence,
by Theorem~\ref{thm:cm-regular},
$\nu = \mu$. Moreover, $\Theta(a,1) = \bigcap\HtM(a) = \bigcap\{\mu\} = \mu$.
This proves (1).

For (2), we have $\HtM(a^*) = \dw{\HtM(a)}^\complement =
\cm{\alg{A}}\setminus \dw{\mu} =
(\um{\alg{A}}\setminus\{\mu\})\cup\{\nu\in\dois{\alg{A}}:\nu\not\subseteq
\mu\}$.

For (3), first note that for any $b\in A$ and any
$\nu\in\um{\alg{A}}$, we have $(b,b^{**})\in\nu$, so
for any $\alpha\in\Con{\alg{A}}$, if $M(\alpha)\subseteq \um{\alg{A}}$
then $(b,b^{**})\in\alpha$. Therefore $\HtM(a^{**})\cap \um{\alg{A}} = \{\mu\}$.
Now, assume $\nu\in \HtM(a^{**})\cap\dois{\alg{A}}$. 
Then, $a^{**}\in 1/\nu\subsetneq 1/\nu^+$, and since
$M(\nu^+)\subseteq\um{\alg{A}}$, we have $a^{**}/\nu^+ = 1/\nu^+ = a/\nu^+$.
Therefore, $M(\nu^+)\subseteq \HtM(a) = \{\mu\}$.
For the converse inclusion, take $\nu\in\cm{\alg{A}}$ with $\nu^+ = \mu$.
Then $a^{**}\in 1/\mu$ and, by Theorem~\ref{thm:orderable-si} we have
$1/\nu^+ = 1/\nu\cup e_\nu/\nu$. If we had $a^{**}\in e_\nu/\nu$, then
$e_\nu/\nu = a^{**}/\nu = (a^{**}/\nu)^{**} = e_\nu^{**}/\nu = 1/\nu$
which is a contradiction. Hence, $a^{**}/\nu = 1/\nu$ and
$\nu\in \HtM(a^{**})$ as required.
\end{proof}

\subsection{Dense and regular elements}
As in the theory of Heyting algebras, $a$ is called \emph{dense} if $a^* = 0$,
and \emph{regular} if $a^{**} = a$. For an algebra $\alg{A}\in\var{Pa}$
we use $R(\alg{A})$ for the set of regular elements of $\alg{A}$,
and $D(\alg{A})$ for the set of dense elements of $\alg{A}$.
Many familiar properties of dense and regular elements carry over to
the theory of p-algebras. We will focus on these that are useful
in relation to completely-meet irreducible congruences.

The usual Glivenko interpretation of Boolean algebras in Heyting algebras
applies without change to p-algebras. Namely, defining $\sim_G$
on any $\alg{A}\in\var{Pa}$ by
$$
a\sim_G b \iff a^{**} = b^{**}
$$
we have that $\sim_G$ is a congruence. Indeed $\sim_G$ turns out to be
the kernel of the map $R\colon \alg{A}\to R(\alg{A})$,
given by $R(a) = a^{**}$. Clearly $R$ is idempotent, and it is straightforward to show
that $R$ preserves $\wedge$ and ${}^*$. In general $R$ does not preserve $\vee$,
but $R(\alg{A}) = \bigl(R(A); \wedge, \sqcup, {}^*, 0,1\bigr)$ with
$x\sqcup y := (x\vee y)^{**}$ is a Boolean algebra.

\begin{lemma}\label{lem:gliv-I}
Let $\alg{A}\in\var{Pa}$. Then $M(\sim_G) = \um{\alg{A}}$.
\end{lemma}  

\begin{proof} For `$\supseteq$' 
let $\mu\in\mathrm{I}_\alg{A}$ and $x\sim_G y$. Then $x^{**} = y^{**}$, and
therefore $x\equiv_{\mu} x^{**} = y^{**} \equiv_{\mu} y$, so $(x,y)\in \mu$.
Hence, $\mathord{\sim_G}\subseteq \mu$, which means $\mu\in M(\sim_G)$, proving
the right-to-left inclusion.

For `$\subseteq$' take $\nu\in M(\sim_G)$ and suppose
$\nu\notin\mathrm{I}_\alg{A}$. Then $\alg{A}/\nu\cong\sipalg{B}$ for 
some Boolean algebra $\alg{B}$. Let $e_\nu$ be an element of $A$ such that
$e_\nu/\nu$ is the unique subcover of $1$ in $\sipalg{B}$.
Since $\sim_G\,\subseteq \nu$, we get
$D(\alg{A}) = 1/_{\!\sim_G}\subseteq 1/\nu$ and $e_\nu\vee e^*_\nu \in
D(\alg{A})$. Therefore 
$(e_\nu\vee e^*_\nu)/\nu = 1/\nu$, whence
$1/\nu = e_\nu/\nu \vee e^*_\nu/\nu = e_\nu/\nu\vee 0/\nu = e_\nu/\nu$ which is
a contradiction. 
\end{proof}

\begin{lemma}\label{lem:gliv-atoms}
Let $\alg{A}\in\var{Pa}$ be finite, with the set of atoms
$\At{\alg{A}} = \{a_1,\dots,a_n\}$.
Then $1/_{\sim_G} = D(\alg{A}) = \up{\bigvee_{i=1}^n a_i}$.
\end{lemma}

\begin{proof}
Let $d_0 \deq  \bigvee_{i=1}^n a_i$. By Lemmas~\ref{lem:homomorphism}
and~\ref{lem:atoms}  we have
$\HtM(d_0) = \bigcup_{i=1}^n \HtM(a_i) = \mathrm{I}_\alg{A}$.
Therefore, $\HtM(d_0^*) = \dw{\HtM(d_0)}^\complement =
(\Cm{\alg{A}})^\complement = \emptyset = \HtM(0)$, which implies $d_0^* = 0$
and so $d_0\in D(\alg{A})$. Now take any $d\in D(\alg{A})$. Since $d^* = 0$, we have
$d/\mu = 1/\mu$ for every $\mu\in\mathrm{I}_\alg{A}$. This means that
$\mathrm{I}_\alg{A}\subseteq \HtM(d)$, and thus $\HtM(d_0)\subseteq\HtM(d)$,
and consequently, by Lemma~\ref{lem:homomorphism}, we obtain $d_0\leq d$ showing
that $d_0$ is the smallest element in $D(\alg{A})$.
\end{proof}

Since by Lemmas~\ref{lem:atoms} and~\ref{lem:atoms-and-cms}(1),
for every finite $\alg{A}\in\var{Pa}$
there is a bijection between  $\At{\alg{A}}$ and $\um{\alg{A}}$,
we introduce another piece of shorthand notation,
and for any $a\in\At{\alg{A}}$ write $\mu_a$ for the unique congruence
$\mu\in\um{\alg{A}}$ such that $\HtM(a) = \{\mu\} = \{\Theta(a,1)\}$.

\begin{lemma}\label{lem:second-storey}
Let $\alg{A}\in\var{Pa}$ be finite, $a\in\At{\alg{A}}$, and let $\nu\in\dois{\alg{A}}$.
If\/ $1/\nu = \up{p}$, then the following hold:   
\begin{enumerate}
\item $\mu_a\notin M(\nu^+) \iff p\leq a^*$, 
\item $\mu_a\in M(\nu^+) \iff\nu < \mu_a\iff a < p$.
\end{enumerate}  
\end{lemma}  

\begin{proof}
For (1), we have 
$$
\mu_a\notin M(\nu^+) \iff \nu\not\subseteq\mu_a
\iff \nu\in\HtM(a^*) \iff a^*\in 1/\nu \iff p\leq a^*
$$
where the second equivalence follows by Lemma~\ref{lem:atoms-and-cms}(2).

For (2), the first equivalence is clear, so we focus on the second.
Forward: we have that $\nu<\mu_a$ implies $1/\nu\subsetneq 1/\mu_a$, so we get
$\up{p}\subsetneq\up{a}$, which in turn implies $a<p$. Backward:
suppose $\nu \not< \mu_a$; then $\mu_a\notin M(\nu^+)$ and so
$p\leq a^*$ by (1); but then $a\leq p \leq a^*$ implying $a=0$ and yielding a
contradiction.  
\end{proof}

The next lemma characterises regular elements of $\alg{A}$ by the behaviour of
congruences from $\cm{\alg{A}}$.

\begin{lemma}\label{lem:tech-for-qid}
Let $\alg{A}\in\var{Pa}$ and $a\in A$. The following are equivalent.
\begin{enumerate}
\item $a^{**} = a$,
\item $\nu\in \HtM(a)\cap\dois{\alg{A}} \iff M(\nu^+)\subseteq
  \HtM(a)\cap\um{\alg{A}}$.
\end{enumerate}  
\end{lemma}

\begin{proof}
Assume (1). The forward implication in (2) is obvious. For the backward
implication, assume $(a,1)\in \nu^+$. By Theorem~\ref{thm:orderable-si} we have 
$1/\nu^+ = 1/\nu\cup e_\nu/\nu$. If $a\in e_\nu/\nu$, then
$1/\nu = (e_\nu/\nu)^{**} = e_\nu^{**}/\nu = a^{**}/\nu = a/\nu = e_\nu/\nu$, so
$(1,e_\nu)\in\nu$, which is a contradiction. Thus, $a\in 1/\nu$, that is
$\nu\in\HtM(a)$; hence $\nu\in \HtM(a)\cap\dois{\alg{A}}$. 

Now assume (2), and suppose $a^{**} > a$. Then for some $\mu\in\dois{\alg{A}}$ 
we have $a^{**}/\mu > a/\mu$. By Lemma~\ref{lem:mu-plus} we know that
$a\in 1/\mu^+$, as otherwise we would have $a/\mu=a^{**}/\mu$.
It follows that $M(\mu^+)\subseteq \HtM(a)\cap\um{\alg{A}}$. 
But it also follows that 
$\mu\notin\HtM(a)\cap\dois{\alg{A}}$, contradicting (2).
\end{proof}

\begin{cor}\label{cor:super-glivenko}
Let $\alg{A}\in\var{Pa}$ and $a,b\in A$. If
$\HtM(a^*)\cap\um{\alg{A}} = \HtM(b^*)\cap\um{\alg{A}}$, then $a^* = b^*$.  
\end{cor}

\begin{lemma}\label{lem:tech-regulars}
Let $\alg{A}\in\var{Pa}$ be finite. Then
$$
R(\alg{A}) = \{\bigl(\bigvee_{a\in S} a\bigr)^{**}: S\subseteq \At{\alg{A}}\}.
$$
\end{lemma}

\begin{proof}
To show that 
$R(\alg{A})\subseteq \{(\bigvee_{a\in S} a)^{**}: S\subseteq \At{\alg{A}}\}$
take $r\in R(\alg{A})$. Using Lemma~\ref{lem:atoms}
put $S\deq \{a\in\At{\alg{A}}: \mu_a\in \HtM(r)\cap\um{\alg{A}}\}$ and note that
by Lemma~\ref{lem:atoms} we have
$$
\HtM(r)\cap\um{\alg{A}} = \bigcup_{a\in S}\HtM(a) \cap\um{\alg{A}} =
\HtM(\bigvee_{a\in S}a)\cap\um{\alg{A}}
= \HtM\bigl((\bigvee_{a\in S}a)^{**}\bigr)\cap\um{\alg{A}}
$$
where the last equality holds
because $(u, u^{**})\in \mu$ for every $\mu\in\um{\alg{A}}$. Then, by
Lemma~\ref{lem:tech-for-qid}(2) we obtain
\begin{align*}
\nu\in\HtM(r)\cap\dois{\alg{A}}
&\iff
M(\nu^+)\subseteq \HtM(r)\cap\um{\alg{A}}\\
&\iff 
M(\nu^+)\subseteq \HtM\bigl((\bigvee_{a\in S}a)^{**}\bigr)\cap\um{\alg{A}}\\
&\iff
\nu\in\HtM\bigl((\bigvee_{a\in S}a)^{**}\bigr)\cap\dois{\alg{A}}  
\end{align*}
Therefore, $\HtM(r) = \HtM\bigl((\bigvee_{a\in S}a)^{**}\bigr)$, and
consequently $r = (\bigvee_{a\in S}a)^{**}$. The backward inclusion is obvious. 
\end{proof}

\section{Free algebras}
Free p-algebras were studied in~\cite{Urq73}, \cite{BD75}, and~\cite{Kat98}
as we already mentioned. We provide yet another description of free p-algebras which
we believe has two major advantages: (i) it is obtained by purely universal
algebraic means, namely, by analysing the structure of completely meet-irreducible
congruences, and (ii) it yields a normal form theorem.

By local finiteness, to obtain a description of all free p-algebras it suffices
to describe the finitely generated ones, so in this section
we fix a finite set $X = \{x_1,\dots,x_k\}$ of free generators (i.e., variables)
and let $\alg{F}_n(k)$ be the p-algebra generated by $X$, free in $\var{Pa}_n$.
By a similar slight abuse of notation we write $\pw(k)$ for $\pw(X)$. 
We begin by describing the atoms of $\alg{F}_n(k)$. 
Let $T\in\pw(k)$, define $f_T:\{x_1,\dots,x_k\}\to \{0,1\}$
putting
\begin{equation}\label{eq:fT}\tag{\P}
f_T(x_i) \deq\begin{cases}
            1 & \text{ if } i\in T\\
            0 & \text{ if } i\notin T,
\end{cases}
\end{equation}
and let $\bar{f}_T$ be the homomorphism onto $\alg{2}$ extending $f_T$.
Obviously for any $T\in\pw(k)$ and any $\mu\in \cm{\alg{F}_n(k)}$ we have
$$  
\mu\in \mathrm{I}_{\alg{F}_n(k)} \iff \mu = \ker{\bar{f}_T}.
$$
Next, define
\begin{equation}\label{eq:at}\tag{at}
x_T \deq \bigwedge_{i\in T}x_i\wedge \bigwedge_{i\notin T} x_i^*.
\end{equation}
The result below, whose first part can also be found in~\cite{BD75},
is immediate from Lemma~\ref{lem:atoms}.

\begin{cor}\label{cor:free-atoms}
For any $T\in\pw(k)$, the element $x_T$ is an atom and
every atom of\/ $\alg{F}_n(k)$ is of this form.
Therefore, if $\mu\in\um{\alg{F}_n(k)}$ then $1/\mu = \up{x_T}$ for some
$T\subseteq k$. Furthermore, $D(\alg{F}_n(k))  = \up{\bigvee_{T\in\pw(k)}x_T}$.
\end{cor}  

By Lemma~\ref{lem:atoms-and-cms}, the set $\HtM(x_T)$ is a singleton,
which according to our notational convention we would write as
$\mu_{x_T}$, but to lighten the notation we will write $\mu_T$.

Next we will 
characterise all non-atomic members of $\J{\alg{F}_n(k)}$.
By Corollary~\ref{cor:cm-into-jirr}, each element $p\in \J{\alg{F}_n(k)}$
is the smallest element of $1/\mu$ for some $\mu \in\cm{\alg{F}_n(k)}$.
Since we already characterised atoms, we can assume that $p$ is such that
$\up{p} = 1/\mu$ for some $\mu\in\dois{\alg{F}_n(k)}$. For this $\mu$ we define
\begin{align}
L &\deq \{i\in k: x_i\in 1/\mu\},\label{eq:aux-L}\tag{L}\\ 
\mathcal{T} &\deq \{T\in\pw(k): \mu \subseteq \mu_T\},\label{eq:aux-T}\tag{T}\\
p^L_\mathcal{T} &\deq \bigl(\bigvee_{T\in\mathcal{T}}x_T\bigr)^{**}\wedge
  \bigwedge_{i\in L} x_i.\label{eq:n-at}\tag{n-at}
\end{align}
Intuitively, $L$ encodes the set of generators that $\mu$ maps to $1$, and
$\mathcal{T}$ encodes the set of maximal congruences
extending $\mu$. Equivalently,
$\bigcap_{T\in \mathcal{T}}\mu_T = \mu^+$. Note that
$\mathcal{T}$ above must satisfy $|\mathcal{T}|\leq n$, as subdirectly
irreducible algebras in $\var{Pa}_n$ has at most $n$ atoms. Consequently,
$\alg{F}_n(k)/\mu \cong \Bnalg{s}$, for $s\leq n$, and so
$\alg{F}_n(k)/\mu^+\cong \alg{2}^s$.  
Note also that $L\subseteq \bigcap\mathcal{T}$, as for any $i\in L$ we have
$x_i\in 1/\mu\subseteq 1/\mu_T = \up{x_T}$, so $i\in T$ for each $T\in \mathcal{T}$.
Then, we obtain the following crucial result.

\begin{lemma}\label{lem:2nd-storey}
  Let $\mu\in\mathrm{II}_{\alg{F}_n(k)}$. Then,
$1/\mu = \up{p_\mathcal{T}^L}$, with $\mathcal{T}$ and $L$ as
  above.
\end{lemma}  

\begin{proof}
As we already noted, 
$\{\mu_T: T\in\mathcal{T}\} = M(\mu^+)$. 
Moreover $0 < |\mathcal{T}|\leq n$, and
for some $p\in\J{\alg{A}}$ we have $1/\mu = \up{p}$

For an arbitrary term $t$, we show that  
$(1,t)\in \mu$ implies $p_\mathcal{T}^L\leq t$. We proceed by induction
on complexity of $t$.
For the base case, that is, if $t = x_i$ for some generator $x_i$,
the claim holds by definition of $p_\mathcal{T}^L$.
The inductive step splits into three cases.
\begin{itemize}  
\item[(i)] If $t = t_1\wedge t_2$, we have immediately $(1,t_1)\in 1/\mu$ and
$(1,t_2)\in 1/\mu$, so by the inductive hypothesis
$p_\mathcal{T}^L\leq t_1$ and $p_\mathcal{T}^L\leq t_2$. Hence,
$p_\mathcal{T}^L\leq t$.
\item[(ii)] If $t = t_1\vee t_2$, then by primeness of $1/\mu$ we get that
$t_1\in 1/\mu$ or $t_2\in 1/\mu$. The inductive hypothesis then yields
$p_\mathcal{T}^L\leq t_1$ or $p_\mathcal{T}^L\leq t_2$, hence
$p_\mathcal{T}^L\leq t$.
\item[(iii)] Finally, assume $t = w^*$. Since $(1,t)\in \mu$, we have
$(1,t)\in\mu_T$ for all $T\in\mathcal{T}$, and so
by Lemmas~\ref{lem:homomorphism} and~\ref{lem:atoms-and-cms}(1) we get  
$\HtM(\bigvee_{T\in\mathcal{T}} x_T) = \{\mu_T: T\in\mathcal{T}\} \subseteq \HtM(t)$.
It follows that
$\bigvee_{T\in\mathcal{T}} x_T\leq t$, and thus
$(\bigvee_{T\in\mathcal{T}} x_T)^{**}\leq t^{**} = w^{***} = w^* = t$.
Hence $p_\mathcal{T}^L\leq t$; in particular $p_\mathcal{T}^L\leq p$.
\end{itemize}

On the other hand, consider 
$p_\mathcal{T}^L/\mu = (\bigvee_{T\in\mathcal{T}}x_T)^{**}/\mu\wedge
\bigwedge_{i\in L} x_i/\mu$. 
By definition of $L$ we have
\begin{equation}\label{eq:mu}\tag{\dag}
\bigwedge_{i\in L} x_i/\mu = 1/\mu.
\end{equation}
Moreover,
\begin{equation}\label{eq:muplus}\tag{\ddag}
1/\mu^+ = \bigcap_{T\in\mathcal{T}} 1/\mu_T = \bigcap_{T\in\mathcal{T}} \up{x_T}
= \up{\bigvee_{T\in\mathcal{T}} x_T}.
\end{equation}
Then, by Lemma~\ref{lem:mu-plus}(2)
$$
1/\mu^+ = 1/\mu\cup e_\mu/\mu 
$$
and by~\eqref{eq:muplus} we have
$\bigvee_{T\in\mathcal{T}} x_T\in 1/\mu^+\setminus 1/\mu$.
Thus, $\bigvee_{T\in\mathcal{T}} x_T\in e_\mu/\mu$, and so
$$
(\bigvee_{T\in\mathcal{T}} x_T)^{**}\in e_\mu^{**}/\mu = 1/\mu. 
$$
Combining this with~\eqref{eq:mu} we obtain
$p_\mathcal{T}^L\in 1/\mu$, so $p_\mathcal{T}^L\geq p$. Therefore
$\up{p_\mathcal{T}^L} = 1/\mu$ as required. 
\end{proof}

\begin{cor}\label{cor:jirr-p}
For every $p\in\J{\alg{F}_n(k)}$, we have
$p = p^L_\mathcal{T}$ for some nonempty $\mathcal{T}\subseteq \pw(k)$ with
$|\mathcal{T}|\leq n$ and some $L\subseteq\bigcap \mathcal{T}$.
\end{cor}

\begin{proof}
For every $\mu\in\cm{\alg{F}_n(k)}$ there exists an element
$p\in\J{\alg{F}_n(k)}$ such that $1/\mu = \up{p}$.
In view of Lemma~\ref{lem:2nd-storey},
all we need to prove is that each $x_T$ defined in~\eqref{eq:at} is of the
desired form. 
Take any $T\subseteq k$ and put $\mathcal{T} = \{T\}$. Then we have
\begin{align*}
p^T_{\mathcal{T}}
&= x_T^{**}\wedge \bigwedge_{i\in T} x_i = 
\Bigl(\bigwedge_{i\in T}x_i\wedge \bigwedge_{i\notin T} x_i^*\Bigr)^{**}
\wedge \bigwedge_{i\in T} x_i\\
&= \bigwedge_{i\in T}x_i^{**}\wedge \bigwedge_{i\notin T} x_i^*
\wedge \bigwedge_{i\in T} x_i =
\bigwedge_{i\notin T} x_i^*\wedge \bigwedge_{i\in T} x_i = x_T
\end{align*}
as claimed.
\end{proof}  

Now, to complete our description of free p-algebras we need to show
that for every $\mathcal{T}\subseteq\pw(k)$ and every
$L\subseteq\bigcap\mathcal{T}$ there exists a $\mu\in\cm{\alg{F}_n(k)}$
such that $M(\mu^+) = \{\mu_T: T\in\mathcal{T}\}$, or equivalently, that there
exist $p\in \J{\alg{F}_n(k)}$ such that $p = p^L_{\mathcal{T}}$. 
By Corollary~\ref{cor:free-atoms} we already know this
for atoms, so let $\mathcal{T}\subseteq \pw(k)$ be such that
$0<|\mathcal{T}|\leq n$. Further, let $L\subseteq \bigcap\mathcal{T}$ be
such that $L\neq\bigcap\mathcal{T}$ if $|\mathcal{T}| = 1$. We will see the
significance of the last condition shortly.  

\begin{lemma}\label{lem:suitable-p}
Let $\mathcal{T}$ and $L$ be as above. Then there exists
$\mu\in\dois{\alg{F}_n(k)}$ such that $M(\mu^+) = \{\mu_T: T\in\mathcal{T}\}$
and $x_i\in 1/\mu \iff i\in L$.  
\end{lemma}

\begin{proof}
Let $\mathcal{T} = \{T_1,\dots, T_s\}$ for $s\leq n$. Consider a map
$g^L_\mathcal{T}\colon \{x_1,\dots,x_k\}\to\Bnalg{s}$ given by
$$
g^L_\mathcal{T}(x_i) \deq\begin{cases}
            1 & \text{ if } i\in L\\
 \bigl(f_{T_1}(x_i),\dots, f_{T_s}(x_i)\bigr) & \text{ if } i\notin L
\end{cases}
$$
with $f_T$ from~\eqref{eq:fT}. Recall that $\Bnalg{s}$ is
$\overline{\alg{2}^s}$ so $g^L_\mathcal{T}$ is well defined.
We claim that $\bar{g}^L_\mathcal{T}\colon \alg{F}_n(k)\to \Bnalg{s}$ is a surjective
homomorphism.

To see it, first assume $s>1$, and then let $j\in\{1,\dots, s\}$
and consider an atom $x_{T_j}$. We have $\bar{g}^L_\mathcal{T}(x_{T_j}) =
(0,\dots,0,1,0,\dots,0)$ with $1$ only at coordinate $j$, so taking joins we see that
every element of $\alg{B}_s$ is in the image of $\bar{g}^L_\mathcal{T}$,
including $e = (1,\dots,1) = \bar{g}^L_\mathcal{T}(\bigvee_{j=1}^{s}x_{T_j})$, since
$s>1$.

The case $s = 1$ has to be treated specially. For if $s=1$, then
$\mathcal{T} = \{T\}$, and to have the range of $\bar{g}^L_\mathcal{T}$ equal
$\{0,e,1\}$ we need at least one $j\in k$ with $g^L_{\{T\}}(x_j) = e$.
By definitions of $g^L_{\mathcal{T}}$ and $f_{T}$, it holds if and only if
$j\in T\setminus L$. 
 
Put $\mu = \ker \bar{g}^L_\mathcal{T}$. Then,
$\mu \in\dois{\alg{F}_n(k)}$ and,
by Lemma~\ref{lem:second-storey}(1,3), we get $M(\mu^+) = \{\mu_{T_1},\dots,\mu_{T_s}\}$.
Obviously, $\bar{g}^L_\mathcal{T}(x_i) = 1$ if and only if $i\in L$,
so $\mu$ is the desired congruence.
\end{proof}  

Apart from its role in the proof above, the condition
$|\mathcal{T}| = 1\implies L\neq\bigcap\mathcal{T}$ ensures that
the appropriate congruence $\mu$ belongs to $\dois{\alg{F}_n(k)}$. Indeed,
taking $\mathcal{T} = \{T\}$ for any $T\subseteq k$, we obtain that
$g^T_\mathcal{T}$ is identical to $f_T$ of \eqref{eq:fT}, and therefore
$\ker\bar{g}^T_{\{T\}} = \ker\bar{f}_T = \mu_T$. From now on, for any 
$\mathcal{T}\subseteq \pw(k)$ with $0<|\mathcal{T}|\leq n$, 
and any $L\subseteq \bigcap\mathcal{T}$, we will write $\mu^L_\mathcal{T}$ for the
appropriate congruence. In particular, our notation gives $\mu^T_{\{T\}} = \mu_T$.

Let us summarise what we know about completely meet-irreducible congruences and
join-irreducible elements.

\begin{cor}\label{cor:cm-and-jirr}
The following hold for the free p-algebra $\alg{F}_n(k)$: 
\begin{enumerate}
\item $\cm{\alg{F}_n(k)} =
  \{\mu^L_{\mathcal{T}}: \mathcal{T}\subseteq \pw(k),\
  0<|\mathcal{T}|\leq n,\ L\subseteq \bigcap\mathcal{T}\}$,
\item $\J{\alg{F}_n(k)} =
  \{p^L_{\mathcal{T}}: \mathcal{T}\subseteq \pw(k),\
  0<|\mathcal{T}|\leq n,\ L\subseteq \bigcap\mathcal{T}\}$.  
\end{enumerate}
\end{cor}  

\begin{example}\label{ex:stone}
Consider $k$-generated free algebra in the variety $\var{Pa}_1$, that is,
in the variety of Stone algebras. We must have
$|\mathcal{T}| = 1$, so let $T\in \pw(k)$ and let $L\subseteq T$. 
Then
\begin{align*}
p^L_{\{T\}} = (x_T)^{**}\wedge \bigwedge_{i\in L} x_i &=
\bigwedge_{j\in T} x_j^{**}\wedge \bigwedge_{j\notin T} x_j^{*}
  \wedge \bigwedge_{i\in L} x_i\\
&= \bigwedge_{i\in L} x_i\wedge\bigwedge_{j\notin T} x_j^{*}\wedge
   \bigwedge_{j\in T\setminus L} x_j^{**}
\end{align*} 
and therefore
$$
\J{\alg{F}_1(k)} = \bigl\{x_1^{\varepsilon(1)}\wedge \dots \wedge x_k^{\varepsilon(k)}:
\text{ for any map } \varepsilon:\{1,\dots,k\}\to \{1,{}^*,{}^{**}\}\bigr\}
$$
where $x^1 = x$. Note that in Boolean algebras, that is in $\var{Pa}_0$,
we have
$$
\J{\alg{F}_0(k)} = \bigl\{x_1^{\varepsilon(1)}\wedge \dots \wedge x_k^{\varepsilon(k)}:
\text{ for any map } \varepsilon:\{1,\dots,k\}\to \{1,{}^*\}\bigr\}.
$$
\end{example}

We are now ready to state the main results of the article.

\begin{theorem}[Structure of free p-algebra]\label{thm:structure}
The following hold:
\begin{enumerate}  
\item $\cm{\alg{F}_n(k)} = \{\mu^L_\mathcal{T}:
  \mathcal{T}\subseteq \pw(k),\ 0<|\mathcal{T}|\leq n,\ 
  L\subseteq \bigcap\mathcal{T}\}$.  
\item $\mu^L_\mathcal{T}\leq^{\Cm{}} \mu^K_\mathcal{S} \iff \mathcal{S}\subseteq
  \mathcal{T} \text{  and  }L\subseteq K$. 
\item $\alg{F}_n(k)\cong \mathrm{Up}(\mathcal{F}_p(\alg{F}_n(k),\subseteq)
  \cong\mathrm{Up}(\cm{\alg{F}_n(k)},\leq^{\Cm{}})
  \cong\mathrm{Up}(\J{\alg{F}_n(k)},\geq)$.
\end{enumerate}
\end{theorem}  

\begin{proof}
Statement (1) follows immediately from Lemma~\ref{lem:suitable-p} and
the remarks following its proof.

For the forward direction of (2), assume $\mu^L_\mathcal{T}\leq^{\Cm{}}
\mu^K_\mathcal{S}$. By Definition~\ref{def:cm-order} we have
$1/\mu^L_\mathcal{T}\subseteq 1/\mu^K_\mathcal{S}$, therefore
$p^K_\mathcal{S}\leq p^L_\mathcal{T}$. To show $\mathcal{S}\subseteq\mathcal{T}$
pick $R\in\mathcal{S}$ and suppose $R\notin\mathcal{T}$. Then
$\mu_R\notin M((\mu^L_\mathcal{T})^+)$ and by
Lemma~\ref{lem:second-storey}(2) we get $p^L_\mathcal{T}\leq
x_R^*$. Hence $x_R\leq \mu^K_\mathcal{S}\leq p^L_\mathcal{T}\leq x_R^*$, which is
a contradiction, showing that $\mathcal{S}\subseteq\mathcal{T}$. Now take $j\in L$.
Then $x_j\in 1/\mu^L_\mathcal{T}\subseteq 1/\mu^K_\mathcal{S}$, so
$x_j\in 1/\mu^K_\mathcal{S}$, and thus $j\in K$. Hence $L\subseteq K$.  

For the backward direction of (2), let
$\mathcal{S}\subseteq\mathcal{T}$ and $L\subseteq K$.
From the former we obtain
$\bigl(\bigvee_{S\in\mathcal{S}} x_S\bigr)^{**}\leq
\bigl(\bigvee_{T\in\mathcal{T}} x_T\bigr)^{**}$, and from the latter we get
$\bigwedge_{i\in K}x_i\leq \bigwedge_{i\in L}x_i$. Hence
$p^K_\mathcal{S}\leq p^L_\mathcal{T}$, so 
$1/\mu^K_\mathcal{S}\subseteq 1/\mu^L_\mathcal{T}$.

In (3) the first isomorphism follows by duality and finiteness, while the second
and third follow from Theorem~\ref{thm:cm-into-fp}, Definition~\ref{def:cm-order} and
Corollary~\ref{cor:cm-and-jirr}. 
\end{proof}

\begin{theorem}[Normal form theorem]\label{thm:normal-form}
Every element $t$ of the algebra $\alg{F}_n(k)$ is of the form
$$  
t = \bigvee\bigl\{p^L_{\mathcal{T}}\in \J{\alg{F}_n(k)}: p^L_{\mathcal{T}}\leq t\bigr\}.
$$
Equivalently,
$$  
t = \bigvee\max\bigl\{p^L_{\mathcal{T}}\in \J{\alg{F}_n(k)}: p^L_{\mathcal{T}}\leq t\bigr\}.
$$
\end{theorem}

\begin{proof}
By previous results, every element of $\J{\alg{F}_n(k)}$ is of the form
$p^{L}_{\mathcal{T}}$  for appropriate $L$ and $\mathcal{T}$.
\end{proof}

\subsection{Digression: free algebras in $\var{H_3}$}
In our description of $\cm{\alg{F}_m(k)}$ we only used the fact that
completely meet-irreducible congruences come in two layers
(cf. Definition~\ref{def:storeys}). The same holds for algebras from the variety
$\var{H_3}$, of Heyting algebras of height 3; in fact we have 
$(\cm{\alg{G}_m(k)},\subseteq) = (\cm{\alg{F}_m(k)},\subseteq)$, where
$\alg{G}_m(k)$ is the free $k$-generated algebra in the variety $\var{H_{3,m}}$.
The parameter $m$ in this case
refers to the \emph{width} of the algebra, defined by the identity
$\bigvee_{i=0}^m(x_i\ra \bigvee_{j\neq i} x_j) = 1$ over the variety of all
Heyting algebras.
The following proposition will be easily proved by the reader,
for example by mimicking our reasoning and using $1$-regularity of Heyting
algebras together with Esakia duality.

\begin{prop}\label{prop:free-H3}
Let $\var{H_{3,m}}$ stand for the variety of Heyting algebras of
height $3$ and width $m$. Let $\alg{G}_m(k)$ be the free
$k$-generated algebra in $\var{H_{3,m}}$. Then,
$$  
\alg{G}_m(k)\cong \mathrm{Up}(\cm{\alg{G}_m(k)},\subseteq).
$$
\end{prop}

By universal algebra, we have that
$\alg{F}_m(k)$ is a subalgebra of the p-algebra reduct $\alg{G}_m(k)^r$ of
$\alg{G}_m(k)$. 
The embedding of $\alg{F}_m(k)$ into $\alg{G}_m(k)^r$ is dual to the following
aesthetically pleasing result, which follows immediately from
Theorem~\ref{thm:cm-summary}(4).

\begin{prop}\label{prop:embed-in-H3}
The identity map on $\cm{\alg{F}_m(k)}$ is a surjective pp-morphism
$$
{id}\colon (\cm{\alg{G}_m(k)},\subseteq)\to
(\cm{\alg{F}_m(k)},\leq^{\Cm{}}).
$$
\end{prop}

\subsection{Examples and some properties}

\begin{example}\label{ex:1-generated}
Consider $k = 1$ so that\/ $\pw(k) = \{\emptyset, \{\emptyset\}\}$.
There are three nonempty subsets of\/ $\pw(k)$, namely
$\mathcal{T}_0 = \{\emptyset\}$, $\mathcal{T}_1 = \{\{\emptyset\}\}$
and $\mathcal{T}_2 = \pw(k)$. Each has as most two elements, so
assume $n\geq 2$. Then, $\bigcap\mathcal{T}_0 = \emptyset =
\bigcap\mathcal{T}_2$ and $\bigcap\mathcal{T}_1 = \{\emptyset\}$.
Therefore, $\J{\alg{F}_n(1)}$ consists of $p_{\mathcal{T}_0}^\emptyset$,
$p_{\mathcal{T}_1}^\emptyset$, $p_{\mathcal{T}_1}^{\{\emptyset\}}$, and
$p_{\mathcal{T}_2}^\emptyset$, and we have two atoms:
$$
x_\emptyset = \bigwedge_{i\in\emptyset}x_i\wedge
\bigwedge_{i\notin \emptyset}x_i^* = 1\wedge x_0^* = x_0^*
$$
and
$$
x_{\{\emptyset\}} = \bigwedge_{i\in \{\emptyset\}}x_i\wedge
\bigwedge_{i\notin\{\emptyset\}} x_i^* = x_0\wedge 1 = x_0.
$$
Putting $x = x_0$ for convenience, we finally obtain
\begin{align*}
p^{\emptyset}_{\mathcal{T}_0} &= x_\emptyset = x^*\\
p^{\{\emptyset\}}_{\mathcal{T}_1} &= x_{\{\emptyset\}}  = x\\
p^{\emptyset}_{\mathcal{T}_1} &= x^{**}\\
p^{\emptyset}_{\mathcal{T}_2} &= (x_\emptyset\vee x_{\{\emptyset\}})^{**}
\wedge \bigwedge_{i\in\emptyset} x_i^*
= (x^*\vee x)^{**}\wedge 1 = 1                                
\end{align*}  
and therefore we have, for $n\geq 1$ 
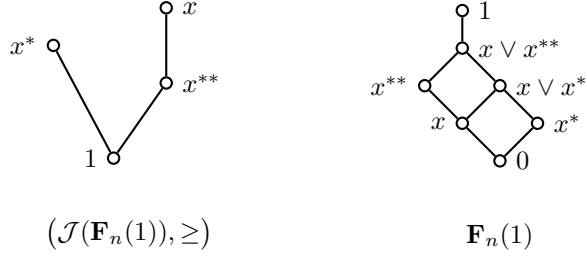
\begin{center}
\begin{tikzpicture}[rotate=180]
\node[dot, label=right:$x$] (v0) at (-0.5,0) {};
\node[dot, label=left:$x^*$] (v1) at (1,0.5) {};
\node[dot, label=right:$x^{**}$] (v2) at (-0.5,1) {};
\node[dot, label=left:$1$] (v3) at (0.2,2) {};
\path[draw,thick] (v0)--(v2)--(v3) (v1)--(v3);
\node (v4) at (0,3) {$\bigl(\J{\alg{F}_n(1)},\geq\bigr)$};
\end{tikzpicture}
\hspace{1.5cm}
\begin{tikzpicture}
\node[dot, label=right:$0$] (v0) at (0,0) {};
\node[dot, label=left:$x$] (v1) at (-0.5,0.5) {};
\node[dot, label=right:{$x^*$}] (v2) at (0.5,0.5) {};
\node[dot, label=left:{$x^{**}$}] (v3) at (-1,1) {};
\node[dot, label=right:{$x\vee x^{*}$}] (v4) at (0,1) {};
\node[dot, label=right:{$x\vee x^{**}$}] (v5) at (-0.5,1.5) {};
\node[dot, label=right:{$1$}] (v6) at (-0.5,2) {};
\path[draw,thick] (v0)--(v1)--(v3)--(v5)--(v6)
(v0)--(v2)--(v4)--(v5) (v1)--(v4);
\node (v7) at (0,-1) {$\alg{F}_n(1)$};
\end{tikzpicture}  
\end{center}
Note that for $n=1$ the set $\mathcal{T}_2$ does not satisfy the requirement
$|\mathcal{T}_2|\leq n$, so we get only the elements
$p^{\emptyset}_{\mathcal{T}_0} = x^*$,
$p^{\emptyset}_{\mathcal{T}_1}  = x^{**}$ and
$p^{\{\emptyset\}}_{\mathcal{T}_1}  = x$. So, we get the
following pictures 
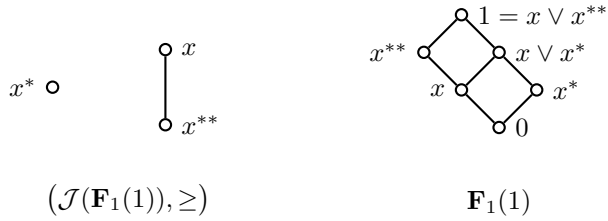
\begin{center}
\begin{tikzpicture}[rotate=180]
\node[dot, label=right:$x$] (v0) at (-0.5,1) {};
\node[dot, label=left:$x^*$] (v1) at (1,1.5) {};
\node[dot, label=right:$x^{**}$] (v2) at (-0.5,2) {};
\path[draw,thick] (v0)--(v2);
\node (v4) at (0,3) {$\bigl(\J{\alg{F}_1(1)},\geq\bigr)$};
\end{tikzpicture}
\hspace{1.5cm}
\begin{tikzpicture}
\node[dot, label=right:$0$] (v0) at (0,0) {};
\node[dot, label=left:$x$] (v1) at (-0.5,0.5) {};
\node[dot, label=right:{$x^*$}] (v2) at (0.5,0.5) {};
\node[dot, label=left:{$x^{**}$}] (v3) at (-1,1) {};
\node[dot, label=right:{$x\vee x^{*}$}] (v4) at (0,1) {};
\node[dot, label=right:{$1 = x\vee x^{**}$}] (v5) at (-0.5,1.5) {};
\path[draw,thick] (v0)--(v1)--(v3)--(v5)
(v0)--(v2)--(v4)--(v5) (v1)--(v4);
\node (v7) at (0,-1) {$\alg{F}_1(1)$};
\end{tikzpicture}  
\end{center}
showing that for $\alg{F}_1(1)$ is isomorphic to
the direct product $\Bnalg{1}\times\Bnalg{0}$.
\end{example}

\begin{example}\label{ex:1-gen-H3}
For contrast, here is the one-generated free algebra $\alg{G}_n(1)$
in $\var{HA_3}$, easy to recognise as a quotient of the Rieger-Nishimura
lattice. Note that $|\J{\alg{G}_n(1)}| = |\J{\alg{F}_n(1)}|$ but the order is
different.   
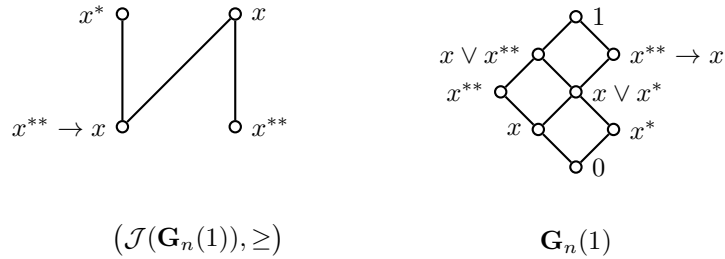
\begin{center}
\begin{tikzpicture}[rotate=180]
\node[dot, label=right:$x$] (v0) at (-0.5,0) {};
\node[dot, label=left:$x^*$] (v1) at (1,0) {};
\node[dot, label=right:$x^{**}$] (v2) at (-0.5,1.5) {};
\node[dot, label=left:$x^{**}\rightarrow x$] (v3) at (1,1.5) {};
\path[draw,thick] (v0)--(v2) (v0)--(v3) (v1)--(v3);
\node (v4) at (0,3) {$\bigl(\J{\alg{G}_n(1)},\geq\bigr)$};
\end{tikzpicture}
\hspace{1.5cm}
\begin{tikzpicture}
\node[dot, label=right:$0$] (v0) at (0,0) {};
\node[dot, label=left:$x$] (v1) at (-0.5,0.5) {};
\node[dot, label=right:{$x^*$}] (v2) at (0.5,0.5) {};
\node[dot, label=left:{$x^{**}$}] (v3) at (-1,1) {};
\node[dot, label=right:{$x\vee x^{*}$}] (v4) at (0,1) {};
\node[dot, label=left:{$x\vee x^{**}$}] (v5) at (-0.5,1.5) {};
\node[dot, label=right:{$1$}] (v6) at (0,2) {};
\node[dot, label=right:{$x^{**}\rightarrow x$}] (v8) at (0.5,1.5) {};
\path[draw,thick] (v0)--(v1)--(v3)--(v5)--(v6)--(v8)--(v4)
(v0)--(v2)--(v4)--(v5) (v1)--(v4);
\node (v7) at (0,-1) {$\alg{G}_n(1)$};
\end{tikzpicture}  
\end{center}  
\end{example}

\begin{example}\label{ex:2-generated}
Similar calculations, which we omit here, produce the picture
of the poset $\bigl(\J{\alg{F}_n(2)},\geq\bigr)$ in
Figure~\ref{fig:2-gener}, where we have
\begin{itemize}
\item $t\deq ((x^*\wedge y)\vee(x\wedge y^*))^{**} = ((x\wedge y)\vee(x^*\wedge y^*))^*$,
\item $u\deq ((x^*\wedge y)\vee(x\wedge y^*)\vee(x^*\wedge y^*))^{**} = (x\wedge y)^*$,
\item $w\deq ((x^*\wedge y)\vee(x\wedge y)\vee(x^*\wedge y^*))^{**} = (x\wedge y^*)^*$,
\item $r\deq ((x\wedge y)\vee(x\wedge y^*)\vee(x^*\wedge y))^{**} = (x^*\wedge y^*)^*$,
\item $s\deq ((x\wedge y)\vee(x\wedge y^*)\vee(x^*\wedge y^*))^{**} = (x^*\wedge y)^*$.
\end{itemize}
Different types of lines indicate $\bigl(\J{\alg{F}_n(2)},\geq\bigr)$ for
$n = 1,2,3,4$. 
Namely, dashed lines (and their incident points) depict $\J{\alg{F}_1(2)}$,
dashed+solid lines depict $\J{\alg{F}_2(2)}$,
dashed+solid+dotted is $\J{\alg{F}_3(2)}$, and the whole picture is
$\J{\alg{F}_4(2)}$. In fact for any $n\geq 4$ (including $n=\omega$)
$\J{\alg{F}_n(2)}$ is the whole picture.

\begin{figure}
\begin{center}  
\begin{tikzpicture}[xscale=2,yscale=1.5,rotate=180]
\tikzset{
every label/.append style={font={\fontsize{8pt}{8pt}\selectfont}},
}

\node[dot, label=above:{$x\wedge y^*$}] (v1) at (-3,-5.5) {};
\node[dot, label=above:{$x\wedge y$}] (v2) at (-1.5,-5.5) {};
\node[dot, label=above:{$x^*\wedge y$}] (v3) at (1,-5.5) {};
\node[dot, label=above:$x^*\wedge y^*$] (v4) at (0,-5.5) {};

\node[dot, label=right:$x^{**}\wedge y^*$] (v5) at (-2.4,-4.2) {};
\node[dot, label=30:$x\wedge y^{**}$] (v6) at (-2,-5) {};
\node[dot, label=120:$x^{**}\wedge y$] (v7) at (-1,-5) {};
\node[dot, label=left:$x^*\wedge y^{**}$] (v8) at (0.5,-4.2) {};
\node[dot, label=right:$x^{**}\wedge y^{**}$] (v9) at (-1.5,-4.5) {};

\node[dot, label=right:$x$] (v10) at (-3.5,-4.2) {};  
\node[dot, label=left:$y$] (v11) at (1.5,-4.2) {};
\node[dot, label=right:$y^*$] (v12) at (-2.4,-3) {};
\node[dot, label=right:$x^{**}$] (v13) at (-3.5,-3) {};
\node[dot, label=left:$t$] (v14) at (-0.2,-3) {};
\node[dot, label=left:$y^{**}$] (v15) at (1.5,-3) {};
\node[dot, label=left:$x^*$] (v16) at (0.5,-3) {};  
\node[dot, label=right:$t^*$] (v17) at (-1.2,-3) {};

\node[dot, label=right:{$s=(x^*\wedge y)^*$}] (v18) at (-3,-1.5) {};
\node[dot, label=right:{$r=(x^*\wedge y^*)^*$}] (v19) at (-1.6,-1.5) {};
\node[dot, label=left:{$u=(x\wedge y)^*$}] (v20) at (-0.4,-1.5) {};
\node[dot, label=left:{$w=(x\wedge y^*)^*$}] (v21) at (1,-1.5) {};

\node[dot, label=below:$1$] (v22) at (-1,-0.5) {};  

\path
(v1) edge[dashed] (v5) (v2) edge[dashed] (v6) (v2) edge[dashed] (v7)
(v3) edge[dashed] (v8) (v9) edge[dashed] (v6) (v9) edge[dashed] (v7)
;

\path
(v10) edge (v1) (v10) edge (v6) 
(v11) edge (v3) (v11) edge (v7)
(v13) edge (v10) (v13) edge (v5) (v13) edge (v9)
(v15) edge (v9) (v15) edge (v8) (v15) edge (v9)
(v15) edge (v11)
(v12) edge (v5) (v12) edge (v4)
(v14) edge (v5) (v14) edge (v8)
(v16) edge (v8) (v16) edge (v4)
(v17) edge (v9) (v17) edge (v4)
;

\path
(v18) edge[dotted] (v12) (v18) edge[dotted] (v13) (v18) edge[dotted] (v17)
(v19) edge[dotted] (v13) (v19) edge[dotted] (v14) (v19) edge[dotted] (v15)
(v20) edge[dotted] (v12) (v20) edge[dotted] (v14) (v20) edge[dotted] (v16)
(v21) edge[dotted] (v17) (v21) edge[dotted] (v15) (v21) edge[dotted] (v16)
;

\path
(v22) edge[dashdotted] (v18) (v22) edge[dashdotted] (v19) (v22) edge[dashdotted] (v20)
(v22) edge[dashdotted] (v21) 
;
\end{tikzpicture}  
\end{center}
\caption{$\bigl(\J{\alg{F}_n(2)},\geq\bigr)$}
\label{fig:2-gener}    
\end{figure}
\end{example}

Note that the identities used in Example~\ref{ex:2-generated} all hold in
$\var{Pa}$ by Corollary~\ref{cor:super-glivenko}. They can be viewed as 
weak De Morgan laws, and greatly simplify several
calculations we omitted.

\begin{prop}
For any finite $k$, consider the algebra $\alg{F}_n(k)$ for any $n\geq 2^k$. We have
$$  
R(\alg{F}_n(k)) = 
\bigl\{\bigl(\bigvee_{T\in\mathcal{T}}x_T\bigr)^{**}: \mathcal{T}\subseteq\pw(k)\bigr\}
\subseteq \J{\alg{F}_n(k)}\cup\{0^{\alg{F}(k)}\}.
$$
In particular, $1^{\alg{F}_n(k)}\in \J{\alg{F}_n(k)}$. 
\end{prop}

\begin{proof}
By Lemma~\ref{lem:tech-regulars} and Corollary~\ref{cor:free-atoms}
we get
$$
R(\alg{F}_n(k)) =
\bigl\{\bigl(\bigvee_{T\in\mathcal{T}}x_T\bigr)^{**}:\mathcal{T}\subseteq\pw(k)\bigr\}
$$
and by Corollary~\ref{cor:jirr-p} we see that
$$
R(\alg{F}_n(k)) = \{p^\emptyset_{\mathcal{T}}: \emptyset\neq
\mathcal{T}\subseteq\pw(k)\}\cup\{0^{\alg{F}_n(k)}\}
\subseteq
\J{\alg{F}_n(k)}\cup\{0^{\alg{F}_n(k)}\}.
$$
\end{proof}

\begin{prop}
For any finite $k$, consider the algebra $\alg{F}_n(k)$ free in $\var{Pa}_n$,
for any $n\geq 2^{k-1}$, with the set $X = \{x_1,\dots,x_k\}$ of free generators.
Then $X\subseteq\J{\alg{F}_n(k)}$.
\end{prop}  

\begin{proof}
Let $i\in\{1,\dots,k\}$, and let $\emptyset\neq \mathcal{T}\subseteq\pw(k)$ and
$L\subseteq\bigcap\mathcal{T}$. Then $\mu^L_{\mathcal{T}}\in\HtM(x_i)$ if and
only if $i\in L$. Hence
$$
\HtM(x_i) = \{\mu^L_{\mathcal{T}}: \emptyset\neq \mathcal{T}\subseteq\pw(k),\
i\in L\subseteq\bigcap\mathcal{T}\}.
$$
Putting $\mathcal{S}
= \bigl\{T\cup\{i\}: T\subseteq \{1,\dots,k\}\setminus\{i\}\bigr\}$,
we see that $|\mathcal{S}|\leq 2^{n-1}$. Then, by
Theorem~\ref{thm:structure}(2) we obtain that
$\mu^{\{x_i\}}_{\mathcal{S}}$ is the smallest element of
$\HtM(x_i)$, which means that $p^{\{x_i\}}_{\mathcal{S}}$
is the greatest element in $\{p^L_{\mathcal{T}}: p^L_{\mathcal{T}}\leq x_i\}$.
Therefore
$$
x_i = \bigvee\{p^L_{\mathcal{T}}: p^L_{\mathcal{T}}\leq x_i\} = p^{\{x_i\}}_{\mathcal{S}} 
$$
proving the claim.
\end{proof}

Since p-algebras are expansions of bounded distributive latices, it 
comes as no surprise that there should be some connection between free
p-algebras and free 
bounded distributive lattices. First note that a free bounded distributive
lattice $\alg{D}(s)$ on finitely many generators has a single atom, so it is naturally
endowed with a p-algebra structure (all elements except $0$ are dense).
The connection we mentioned is stated below.

\begin{theorem}\label{thm:free-distrib}
Let $\alg{F}_n(k)$ be the $k$-generated free algebra in $\var{Pa}_n$, and
let $T\in\pw(k)$. Then
$$
\alg{F}_n(k)/\Theta(1,x_T^{**})\cong \alg{D}(s)
$$  
where $s = |T|$ and $\alg{D}(s)$ is viewed as a p-algebra.
\end{theorem}
  
\begin{proof}
First we consider the case $T=\emptyset$. Then $x_T^{**} = x_T$, 
so in this particular case $x_T^{**}$ is an atom, and hence by
Lemma~\ref{lem:atoms-and-cms}(1) we have
$\Theta(1,x_T^{**})=\mu_T\in\um{\alg{F}_n(k)}$. Therefore,
$\alg{F}_n(k)/\Theta(1,x_T^{**})\cong\alg{2} = \alg{D}(0)$.

Now assume $T\neq\emptyset$ and let $Y = \{y_i: i\in T\}$. Define
$h\colon \{x_1,\dots,x_k\}\to D(s)$ putting
$$
h(x_i) = \begin{cases}
  y_i & \text{ if } i\in T\\
  0 & \text{ if } i\notin T
  \end{cases}
$$
so that $\overline{h}\colon \alg{F}_n(k)\to \alg{D}(s)$ is a surjective
(since $Y$ is contained in the image of $\overline{h}$) p-algebra homomorphism.
Moreover, we have
\begin{align*}
\overline{h}(x_T^{**}) &= 
\overline{h}\bigl((\bigwedge_{i\in T} x_i\wedge
  \bigwedge_{i\notin T} x_i^*)^{**}\bigr)
=  \bigwedge_{i\in T} \overline{h}(x_i)^{**}\wedge
                         \bigwedge_{i\notin T} \overline{h}(x_i)^*\\
&= \bigwedge _{i\in T} y_i^{**} \wedge 1 = 1 
\end{align*}  
and therefore $\Theta(1,x_T^{**})\subseteq\ker \overline{h}$.

To show the converse, note first that since $T\neq \emptyset$ we have
$\HtM(x_T^{**})\cap\dois{\alg{F}_n(k)}\neq\emptyset$ and thus by
Lemma~\ref{lem:atoms-and-cms}(3) we get
$$
\Theta(1,x_T^{**}) = \bigcap\{\nu\in\dois{\alg{F}_n(k)}: \nu^+ = \mu_T\}
$$
and using the description of $\cm{\alg{F}_n(k)}$ from
Theorem~\ref{thm:structure}, we obtain
$$
\Theta(1,x_T^{**}) = \bigcap\{\mu^L_{\{T\}}:L\subsetneq T\}.
$$
Let us now take $L\subsetneq T$ and a map $\pi_L: Y\to \Bnalg{0}$ given by
$\pi_L(y_i) =\begin{cases}
  1 & \text{ if } i\in L,\\
  e & \text{ if } i\in T\setminus L.
\end{cases}$
Then, as $\overline{\pi}_L(0) = 0$, we have that  $\overline{\pi}_L$ is
a surjective bounded distributive lattice homomorphism. Moreover,
since $\overline{\pi}_L(a)\in \{1,e\}$ for any $a\in D(s)\setminus\{0\}$,
it is easy to verify that $\overline{\pi}_L$ is also a p-algebra homomorphism.
From the proof of Lemma~\ref{lem:suitable-p} we know that
$\mu^L_{\{T\}} = \ker \overline{g}^L_{\{T\}}$, where
$g^L_{\{T\}}\colon \{x_1,\dots,x_k\}\to\Bnalg{0}$ is given by
$$
g^L_{\{T\}}(x_i) \deq\begin{cases}
            1 & \text{ if } i\in L\\
     f_T(x_i) & \text{ if } i\notin L
\end{cases}
$$
with $f_T$ from~\eqref{eq:fT}. Unwrapping the definition for the present case  
gives
$$
g^L_{\{T\}}(x_i) =\begin{cases}
  1 & \text{ if } i\in L\\
  e & \text{ if } i\in T\setminus L\\
  0 & \text{ if } i\notin T
\end{cases}
$$
Now it is easy to see that $\overline{\pi}_L\circ \overline{h}|_{\{x_1,\dots,x_k\}}
= \overline{g}^L_{\{T\}}|_{\{x_1,\dots,x_k\}}$, so by the universal property of the free
algebra we obtain $\overline{\pi}_L\circ \overline{h} = \overline{g}^L_{\{T\}}$
and this implies that
$\ker\overline{h}  \subseteq \ker \overline{g}^L_{\{T\}}  = \mu^L_{\{T\}}$.
Hence $\ker\overline{h}\subseteq  \bigcap\{\mu^L_{\{T\}}:L\subsetneq T\}
= \Theta(1,x_T^{**})$ as claimed. 
\end{proof}

From Theorem~\ref{thm:free-distrib} we can easily derive the fact (well known
before, see, e.g.,~\cite{Urq73}) that a free finitely-generated Stone algebra is a
direct product of finite free bounded distributive lattices (viewed as
p-algebras).

\begin{cor}\label{cor:stone-a}
For any $k$, we have
$$
\alg{F}_1(k)\cong
\prod_{s=0}^k\alg{D}(s)^{\binom{k}{s}}
$$
where $\alg{D}(s)$ is the free bounded distributive lattice on $s$ generators,
viewed as a p-algebra.
\end{cor}

\begin{proof}
By the construction of free algebras for $\var{Pa}_1$ (see
Theorem~\ref{thm:structure}) we know that
$$
(\cm{\alg{F}_1(k)};\leq^{\Cm{}}) =
\biguplus_{T\in\pw(k)}\bigl(\{\mu^L_{\{T\}}: L\subseteq T\};\leq^{\Cm{}}\bigr).
$$ 
Hence, putting $\theta_T \deq \Theta(x_T^{**},1)$ for any $T\in\pw(k)$,
by Lemma~\ref{lem:atoms-and-cms}(3) we obtain, 
$$
(\cm{\alg{F}_1(k)/\theta_T};\leq^{\Cm{}}) \cong
\bigl(\{\mu^L_{\{T\}}: L\subseteq T\};\leq^{\Cm{}}\bigr).
$$
By Theorem~\ref{thm:cm-summary}(4)(a) we thus have
$$
\mathrm{Up}\bigl(\{\mu^L_{\{T\}}: L\subseteq T\};\leq^{\Cm{}}\bigr)
\cong \alg{F}_1(k)/\theta_T
$$
and therefore by duality we obtain
$$
\alg{F}_1(k)\cong \mathrm{Up}(\cm{\alg{F}_1(k)};\leq^{\Cm{}})
\cong \prod_{T\in\pw(k)}\alg{F}_1(k)/\theta_T
$$
from which by Theorem~\ref{thm:free-distrib} we immediately get
$$
\alg{F}_1(k)\cong
\prod_{s=0}^k\alg{D}(s)^{\binom{k}{s}}.
$$
\end{proof}

We can also quite easily obtain a formula for counting the number of elements in
$\J{\alg{F}_n(k)}$. The formula was known already to Berman
and Dwinger, see~\cite{BD75}, but in our version  
it amounts to the following.
Fix $L\subseteq\{1,\dots,k\}$, let $\ell = |L|$. Then we choose supersets
of $L$ so that their number does not exceed $n$. Let the number of these
supersets be $m$. The number of choices of such supersets is
$\binom{2^{k-\ell}}{m}$. Then by simple combinatorics we get the
formula below.

\begin{prop}\label{prop:counting}
For any $k\geq 0$ we have
\begin{equation*}\label{eq:J}\tag{J}
|\J{\alg{F}_n(k)}|
= \sum_{\ell=0}^k \binom{k}{\ell}\sum_{m=1}^n\binom{2^{k-\ell}}{m}. 
\end{equation*}
\end{prop}

\begin{example}\label{exa:count}
We invite the reader to verify that for $n\geq 2^k$ the formula~\eqref{eq:J}
simplifies to   
$$
|\J{\alg{F}_n(k)}| = \sum_{\ell=0}^k \binom{k}{\ell}(2^{2^{\ell}} -1). 
$$
and that in full generality, we have
$$
|\J{\alg{F}_1(k)}| = 3^k \qquad\text{ and }\qquad |\J{\alg{F}_2(k)}| =
\frac{1}{2}(5^k+3^k). 
$$
\end{example}  

\subsection{Structural completeness}
A deductive system $\mathcal{S}$ is \emph{structurally complete}
if all its admissible rules are derivable. 
The notion of structural completeness was introduced by Pogorzelski
in~\cite{Pog71}, and later investigated by many authors in various forms and
modifications. We refer the reader to~\cite{PW08} for a reasonably up-to-date
monograph of structural completeness. Deductive systems correspond,
roughly, to quasivarieties, so an analogous notion naturally arises for
them. Algebraic investigations of structural completeness have enjoyed a revival
in the 21st Century, see for instance~\cite{ORvA}, \cite{CM09}, \cite{RS16},
\cite{BM23} --- or~\cite{DS16}, \cite{Wro09} for certain weakenings of the concept.

Structural completeness of a 
quasivariety $\mathcal{Q}$ was characterised in~\cite{Ber91}, namely,  
$\mathcal{Q}$ is structurally complete if and only if $\mathcal{Q}$ is generated
by its free algebra on countably many generators, or equivalently,
by its finitely generated free algebras. If $\mathcal{V}$ is a variety, then
$\mathcal{V}$ is structurally complete if and only if $\mathcal{V}$ is generated by its
free algebras as a quasivariety. 

Thus, to prove failure of structural completeness for a variety $\mathcal{V}$
it suffices to exhibit a quasiequation that holds in free $\mathcal{V}$-algebras
but fails in $\mathcal{V}$. Conversely, to establish structural completeness of
$\mathcal{V}$ it suffices to show that all subquasivarieties of
$\mathcal{V}$ are varieties. For if there were a quasiequation $\mathbf{q}$
such that $\alg{F}_{\mathcal{V}}(X)\models\mathbf{q}$ but
$\mathcal{V}\not\models\mathbf{q}$, then
$\mathrm{Mod}(\mathbf{q})\cap\mathcal{V}$ would be a proper subquasivariety of
$\mathcal{V}$ which is not a variety, as
$\alg{F}_{\mathcal{V}}(X)\in\mathrm{Mod}(\mathbf{q})\cap\mathcal{V}$ but
$H(\alg{F}_{\mathcal{V}}(X))\subsetneq\mathrm{Mod}(\mathbf{q})\cap\mathcal{V}$.
In fact the same reasoning applies to all subvarieties of $\mathcal{V}$, so
if all subquasivarieties of $\mathcal{V}$ are varieties, then all of
them are structurally complete, or as we say $\mathcal{V}$ is
\emph{hereditarily structurally complete}. 

The following quasi-identities will be used in a sequel article to obtain a few
new results about the structure of the lattice of subquasivarieties of
$\var{Pa}$. Here we will use only one of them to prove structural
incompleteness for most of varieties of p-algebras.  
\begin{equation*}
\bigFOand_{1\leq i\leq n} (x^*_i = \bigvee_{j\neq i} x_j)
\tag{$\mathbf{qb}_n$}\ \implies \bigvee_{1\leq i\leq n} x_i = 1.
\end{equation*}
Our description of free algebras quickly yields the next lemma.

\begin{lemma}\label{lem:b3-true-in-Fm}
Let $\mathbf{F}_m(X)$ be a free algebra in
$\var{Pa}_m$, for $m\geq 3$ and $X\neq\emptyset$.
Then $\mathbf{F}_m(X)\models \mathbf{qb}_3$.
\end{lemma}

\begin{proof}
To get a contradiction, suppose $\mathbf{F}_m(X)\not\models \mathbf{qb}_3$.
Then there exist $t_1,t_2,t_3\in \mathbf{F}_m(X)$ such that
$\forall i\in\{1,2,3\}: t_i^* = \bigvee_{j\neq i} t_j$, and
$t_1\vee t_2\vee t_3\neq 1$. Note that then for all $i,j\in\{1,2,3\}$ we have
$$
i\neq j \implies t_i\wedge t_j = 0
$$
as $0 = t_i\wedge t_i^* = t_i\wedge (\bigvee_{j\neq i} t_j) =
\bigvee_{j\neq i}(t_i\wedge t_j)$. Moreover, $t_1\vee t_2 \vee t_3 =
t_1\vee t_1^*$ is a dense element of $\alg{F}_m(X)$, that is,
$(t_1\vee t_2 \vee t_3)^{**} = 1$. Since $t_1\vee t_2 \vee t_3 < 1$, there
exists $\mu\in\dois{\alg{F}_m(X)}$ such that
$(t_1\vee t_2 \vee t_3)/\mu = e_\mu/\mu$. We will show that
\begin{equation*}\label{eq:ast}\tag{$\ast$}
\forall \varphi\in\um{\alg{F}_m(X)}: \varphi > \mu
\implies \exists!\; i: (1,t_i)\in \varphi.
\end{equation*}
To see it, take $\varphi \in\um{\alg{F}_m(X)}$ with $\varphi>\mu$, and
note that $(t_1^*\wedge t_2^* \wedge t_3^*,0)\in\mu$. Thus,
$(t_1^*\wedge t_2^* \wedge t_3^*,0)\in\varphi$; therefore, as
$\varphi>\mu$, we must have $(1,t_i)\in\varphi$
for some $i\in\{1,2,3\}$. Since
for all $i\neq j$ we have $t_i\wedge t_j = 0$, we obtain
$(0,t_j)$ for all $j\neq i$, showing~\eqref{eq:ast}.

On the other hand, for each $i\in\{1,2,3\}$
there exists $\varphi_i\in\um{\alg{F}_m(X)}$ such that
$\varphi_i>\mu$ and $(1,t_i)\in\varphi_i$. To show it, assume otherwise, that
is, that we have $(1,t_i)\notin\varphi_i$ for all
$\varphi\in\um{\alg{F}_m(X)}$ such that $\varphi>\mu$. This implies
$(0,t_i)\in\mu^+$, whence $(0,t_i)\in\mu$ as by Lemma~\ref{lem:mu-plus}(3) we
have $0/\mu^+ = 0/\mu$. But this further implies $t_i^*/\mu =
1/\mu$, which is a contradiction.

Now take $\varphi_1,\varphi_2\in\um{\alg{F}_m(X)}$ such that
$(1,t_1)\in\varphi_1$, $(1,t_2)\in\varphi_2$ and
$\varphi_1\wedge\varphi_2>\mu$. Then, by~\eqref{eq:ast}, we
obtain $(1,t_3)\notin\varphi_1$ and $(1,t_3)\notin\varphi_2$. Hence
$(1,t_3^*)\in \varphi_1\wedge\varphi_2$.
By the construction of $\alg{F}_n(k)$ we obtain that
for some $\gamma\in\dois{\alg{F}_n(k)}$ we have $\gamma^+
= \varphi_1\cap\varphi_2$. By Lemma~\ref{lem:mu-plus}(2) this implies
$t_3^*\in 1/\gamma^+ = 1/\gamma\cup e_\gamma/\gamma$. Obviously,
$t_3^*/\gamma \neq e_\gamma/\gamma$; therefore
$t_3^*/\gamma = 1/\gamma = (t_1\vee t_2)/\gamma$.
It follows that $(1,t_1)\in\gamma\leq \varphi_2$ or
$(1,t_2)\in\gamma\leq \varphi_1$, which yields a contradiction in either case.
\end{proof}

\begin{theorem}\label{thm:str-incompl}
For $m < 3$ the variety $\var{Pa}_m$ is hereditarily structurally complete.
For $m\geq 3$ the variety $\var{Pa}_m$ is structurally incomplete.
\end{theorem}

\begin{proof}
Let $\mathcal{Q}$ be nontrivial a quasivariety of p-algebras. We claim that
$\mathcal{Q}\subseteq \var{Pa}_2$, then
$\mathcal{Q}\in\{\var{Pa}_{0}, \var{Pa}_{1}, \var{Pa}_{2}\}$.
Since $\mathcal{Q}$ is nontrivial, we have
$\Bnalg{0}\in\mathcal{Q}$ and so $\var{Pa}_0\subseteq\mathcal{Q}$.
Suppose there is a $\alg{C}\in \mathcal{Q}\setminus\var{Pa}_0$. Then
$\alg{C}$ is not a Boolean algebra, and so there is a $c\in C$ such that
$c\vee c^*\neq 1$. Since $c\vee c^* > 0$, it is easy to see that
$\{0,c\vee c^*,1\}$ is a subuniverse of $\alg{C}$, and the p-algebra on
$\{0,c\vee c^*,1\}$ is isomorphic to $\Bnalg{1}$. Hence $\var{Pa}_1\subseteq
\mathcal{Q}$. Now assume $\var{Pa}_1\subsetneq \mathcal{Q}$. Since
$\var{Pa}_1$ is axiomatised by the identity $x^*\vee x^{**} = 1$, 
there is an algebra $\alg{D}\in \mathcal{Q}$ and an element $d\in D$
such that  $d^{*}\vee d^{**}\neq 1$. Clearly, $d^*, d^{**}\notin\{0,1\}$.
It is again easy to verify that $\{0,d^*, d^{**}, d^*\vee d^{**}, 1\}$
is a subalgebra of $\alg{D}$ isomorphic to $\Bnalg{2}$. Hence
$\mathcal{Q} = \var{Pa}_2$. This proves the first statement.

To see that $\mathbf{qb}_n$ fails on $\Bnalg{3}$, evaluate $x_1$,
$x_2$, $x_3$ to the three different atoms. Then,
since $\Bnalg{3} \in\var{Pa}_m$ for all $m \geq 3$, the second statement holds
as well.
\end{proof}

\bibliographystyle{plain}
\bibliography{p-algebras}

\end{document}